\documentclass[reqno]{amsart}
\usepackage{amsmath, amsthm, amssymb, amstext, mathrsfs}

\usepackage[left=3cm,right=3cm,top=2cm,bottom=2cm,includeheadfoot]{geometry}
\usepackage{hyperref,xcolor}% http://ctan.org/pkg/{hyperref,xcolor}
\hypersetup{
	pdfborder={0 0 0},
	colorlinks,
}
\usepackage{enumitem}
\usepackage{amsfonts}
\usepackage{amsthm}

\usepackage{amsfonts}
\usepackage{amsthm}

\DeclareMathOperator*{\supp}{supp}

\newcommand*\diff{\mathop{}\!\mathrm{d}}

\renewcommand{\l}{\left}
\renewcommand{\r}{\right}

\newcommand{\norm}[1]{|\!|#1|\!|}

\newcommand{\round}[1]{\left(#1\right)}

\newcommand{\scal}[1]{\left\langle#1\right\rangle}

\newcommand{\ol}{\overline}
\newcommand{\R}{{\mathbb R}}
\newcommand{\N}{{\mathbb N}}

\newcommand{\Lp}[1]{L^{#1}(\Omega)}

\newcommand{\Wp}[1]{W^{1,#1}(\Omega)}

\newcommand{\Wpzero}[1]{W^{1,#1}_0(\Omega)}

\newcommand{\close}{\overline{\Omega}}

\newcommand{\eps}{\varepsilon}

\newcommand{\into}{\int_{\Omega}}

\newcommand{\Assg}[1]{\textup{(g)}}

\begin{document}

	\title[Multiplicity results for critical double phase problems]
	{Multiplicity results for double phase problems involving a new type of critical growth}

	\author[H.H.\ Ha and K.\ Ho]{Hoang Hai Ha and Ky Ho}
	
	\address{Hoang Hai Ha \newline
		Faculty of Applied Science, Ho Chi Minh City University of Technology,
		Vietnam National University Ho Chi Minh City, 268 Ly Thuong Kiet St., Dist. 10, Ho Chi Minh City, Vietnam}
	\email{hoanghaiha1910@gmail.com }
	
	\address{Ky Ho \newline
		Department of Mathematical Sciences, Ulsan National Institute of Science and Technology, 50 UNIST-gil, Eonyang-eup, Ulju-gun, Ulsan 44919, Republic of Korea}
	\email{kyho@unist.ac.kr, hnky81@gmail.com}

	%------------------------------------------------------------------------------------------------------------------------------------
	\begin{abstract}
		
		Using variational methods, we obtain several multiplicity results for double phase problems that involve variable exponents and a new type of critical growth. This new critical growth is better suited for double phase problems when compared to previous works on the subject. In order to overcome the lack of compactness caused by the critical exponents, we establish a concentration-compactness principle of Lions type for spaces associated with double phase operators, which is of independent interest to us. Our results are new, even in the case of constant exponents.

	\end{abstract}
	\maketitle \numberwithin{equation}{section}
	\newtheorem{theorem}{Theorem}[section]
	\newtheorem{lemma}[theorem]{Lemma}
	\newtheorem{definition}[theorem]{Definition}
	\newtheorem{claim}[theorem]{Claim}
	\newtheorem{proposition}[theorem]{Proposition}
	\newtheorem{remark}[theorem]{Remark}
	\newtheorem{corollary}[theorem]{Corollary}
	\newtheorem{example}[theorem]{Example}
	\newtheorem{assumption}[theorem]{Assumption}
	\allowdisplaybreaks

	%=========================INTRODUCTION============================== %
	\section{Introduction}\label{Intro}
	
	In this paper, we investigate the multiplicity of solutions to the following double phase problem involving critical growth
\begin{equation}\label{e1.1}
	\begin{cases}
		- M\left(\int_\Omega \mathcal{A}(x,\nabla u)\diff x\right)\operatorname{div}A(x,\nabla u) = \lambda f(x,u) +\theta B(x,u)& ~\text{in}~ \Omega ,\\
		u=0& ~\text{on} ~ \partial \Omega ,
	\end{cases}
\end{equation}
where $\Omega $ is a bounded domain of $\mathbb{R}^{N}$ with Lipschitz boundary $\partial \Omega$; $\mathcal{A}: \overline{\Omega }\times\R^N\to \R$ and $A: \overline{\Omega }\times\R^N\to \R^N$ are given by
\begin{equation*}\label{A1}
	\mathcal{A}(x,\xi):=\frac{1}{p(x)}|\xi|^{p(x)}+\frac{a(x)}{q(x)}|\xi|^{q(x)}
\end{equation*} 
and 
$$A(x,\xi):=\nabla_\xi \mathcal{A}(x,\xi)=|\xi|^{p(x)-2}\xi+a(x)|\xi|^{q(x)-2}\xi$$ 
for all $(x,\xi)\in \overline{\Omega}\times\R^N$ with $a,p,q\in C^{0,1}(\overline{\Omega })$ such that $a(\cdot) \geq 0,\, 1<p(\cdot)<q(\cdot)<N$; $M: \, [0, \infty) \to \R$ is a real function; $f,B:\Omega\times \mathbb{R} \to \mathbb{R}$ are  Carath\'eodory functions in which $f(x,u)$ has a subcritical growth while $B(x,u)$ generalizes the critical term $c_1|u|^{p^*(x)-2}u+c_2a(x)^{\frac{q^*(x)}{q(x)}}|u|^{q^*(x)-2}u$ with  $c_1>0$, $c_2\geq 0$, and 
$h^*(\cdot):=\frac{Nh(\cdot)}{N-h(\cdot)}$ for $h(\cdot)<N$; and $\lambda,\theta$ are positive real parameters.

	\vskip5pt
In the past two decades, there has been significant focus on addressing problems associated with the double phase operator $\operatorname{div}\left(|\nabla u|^{p-2}\nabla u+a(x)|\nabla u|^{q-2}\nabla u\right)$. The research idea originated with Zhikov \cite{Z86,Z11} when he explored Lavrentiev's phenomena and the elastic theory. The problems relevant to the double phase operator  have led to the study of the energy  functional given by
\begin{equation}\label{Func1}
	I(u):=\int_\Omega \left( \frac{1}{p}|\nabla u|^{p}+\frac{a(x)}{q}|\nabla u|^{q}\right) \diff x.
\end{equation}
A natural extension of \eqref{Func1} is the functional 
\begin{equation}\label{Func2}
	J(u):= \int_\Omega \mathcal{A}(x,\nabla u)\diff x=\int_\Omega \left(\frac{1}{p(x)}|\nabla u|^{p(x)}+\frac{a(x)}{q(x)}|\nabla u|^{q(x)} \right)\diff x.
\end{equation}
The  integrand $\mathcal{A}(x,\xi)$ of this functional has the unbalance growths when $0 \leq a(\cdot) \in L^\infty(\Omega)$, that is
\begin{align*}
	c_1|\xi|^{p(x)} \leq \mathcal{A}(x,\xi) \leq c_2 \l(1+|\xi|^{q(x)}\r)
\end{align*}
for a.\,a.\,$x\in\Omega$ and for all $\xi\in\R^N$ with $c_1,c_2>0$.

The term ``double phase" comes from the continuous switching between two phases of elliptic behavior within the integrand $\mathcal{A}$, with this process being controlled by the weighted function $a(\cdot)$. That means the energy density of $J$ demonstrates ellipticity in the gradient of order $q(\cdot)$ in the set $\{x\in\Omega\,:\,a(x)>\eps\}$ for any fixed $\eps>0$, and the  order $p(\cdot)$ on the set $\{x\in\Omega\,:\,a(x)=0\}$.

Applications of the double phase differential operators $\operatorname{div}\left(|\nabla u|^{p(x)-2}\nabla u+a(x)|\nabla u|^{q(x)-2}\nabla u\right)$ ($=\operatorname{div}A(x,\nabla u)$) and energy functional in \eqref{Func2} can be found in physics and engineering. For instance, Bahrouni-R\u{a}dulescu-Repov\v{s} \cite{Bah-Rad-Rep.19} studied transonic flows, Benci-D'Avenia-Fortunato-Pisani \cite{BDFP00} explored quantum physics, and Cherfils-Il’yasov \cite{CI05} investigated reaction-diffusion systems. Notably, in elasticity theory, the modulating coefficient $a(\cdot)$ governs the geometry of composites made of two materials with distinct power hardening exponents $q(\cdot)$ and $p(\cdot)$, as described by Zhikov \cite{Z11}. As a result, there have been a number of works on double phase problems from various perspectives. 

The regularity of local minimizers for double phase functionals have been a center of attention since the seminal work by Marcellini \cite{Mar89b}. On this topic, we can refer the reader to the works \cite{BCM15,BCM16,BCM18,BKM15,BO20,CM15a,CM15b,DeF18,HHT17,Mar89b,Mar91,Ok18,Ok20,RT16,RT20} and others, as well as the references therein. Existence results for double phase problems have also intensively been studied in recent years. Among others, we refer to works \cite{CS16,FW21,GP19,GW20a,GW20b,GW21,LD18,LD20,PS18} for subcritical problems and \cite{AFMW22,FFW22,MW19} for problems involving the critical growth $p^\ast$. 

\vskip5pt
 In order to study problems driven by $\operatorname{div}A(x,\nabla u)$, Crespo-Blanco-Gasi\'{n}ski-Harjulehto-Winkert presented the fundamental properties of Musielak-Orlicz spaces $L^{\mathcal{H}}(\Omega)$ and Musielak-Orlicz Sobolev spaces $W^{1,\mathcal{H}}(\Omega)$ in \cite{CGHW}, where
\begin{equation*}
	\mathcal{H}(x,t):=t^{p(x)}+a(x)t^{q(x)} \text{ for } (x,t) \in \overline{\Omega} \times [0,\infty)
\end{equation*}
(see Section~\ref{Pre} for the definitions of the Musielak-Orlicz space $\Lp{\mathcal{H}}$ and the Musielak-Orlicz-Sobolev space $\Wp{\mathcal{H}}$). Furthermore, \cite{CGHW} presents the existence and uniqueness of solutions for a class of problems involving $\operatorname{div}A(x,\nabla u)$. Subsequently, several studies investigating $\operatorname{div}A(x,\nabla u)$ have been conducted, and interested readers can find references to some of these studies in \cite{CGHW,KKOZ22,VW22}. It is worth pointing out that in both constant and variable exponent cases, while existing works treat double phase problems in terms of two exponents $p(\cdot)$ and $q(\cdot)$ with $p(\cdot)<q(\cdot)$, the reaction terms in those works have a growth that does not exceed $p^*(\cdot)$. It is clear that such a growth does not capture the behavior for the critical growth of Sobolev-type embedding on $\Wp{\mathcal{H}}$, which continuously switches between $p^*(\cdot)$ and $q^*(\cdot)$. %This is because the implicit definition of $\mathcal{H}_*$ has prevented a condition of the nonlinear term from being compatible with the embedding $\Wp{\mathcal{H}}\hookrightarrow \Lp{\mathcal{H}_*}$. 
To fill this gap, Ho and Winkert \cite{HW22} recently proposed the following critical embedding on $\Wp{\mathcal{H}}$:
\begin{equation}\label{I.CE}
\Wp{\mathcal{H}}\hookrightarrow \Lp{\mathcal{G}^\ast},
\end{equation}
where
$$\mathcal{G}^\ast(x,t):=t^{p^\ast(x)}+a(x)^{\frac{q^\ast(x)}{q(x)}}t^{q^\ast(x)} \text{ for } (x,t) \in \overline{\Omega} \times [0,\infty).$$ 
They also showed the optimality of $\mathcal{G}^\ast$ among those  $\Psi$ of the form $\Psi(x,t)=|t|^{r(x)}+a(x)^{\frac{s(x)}{q(x)}}|t|^{s(x)}$ such that $
W^{1,\mathcal{H}}(\Omega)
\hookrightarrow  L^{\Psi}(\Omega)$. As an application of the critical embedding \eqref{I.CE}, in \cite{HW22} the authors obtained the boundedness of solutions to generalized double phase problems. To the best of our knowledge, however, no existence results have been established for double phase problems with growth exceeding $p^*(\cdot)$.  Motivated by this, we investigate the multiplicity results for problem \eqref{e1.1}, and this is the first attempt to handle double phase problems with growth surpassing the threshold of $p^*(\cdot)$. It is noteworthy that problem~\eqref{e1.1} contains not only critical exponents but also a nonlocal Kirchhoff term. The Kirchhoff problem, initially introduced by Kirchhoff \cite{Kir1876}, originated from the study of an extension of the classical D'Alembert's wave equation, specifically concerning the analysis of free vibrations in elastic strings. Elliptic problems of Kirchhoff type have a rich background in several physics applications and have gained significant attention from numerous researchers in recent years, e.g., \cite{AFMW22, APS10,DS92,FMPV23,FP20,HMR17} and the references therein.

Critical problems have their origins in geometry and physics, and the most notorious example is Yamabe's problem \cite{Aub76}. After the seminal work by Brezis-Nirenberg \cite{BN83} on the Laplace equation, numerous extensions and generalizations have been made in various directions. The lack of compactness that arises in connection with the variational approach makes the critical problems delicate and interesting. For such problems, the concentration-compactness principle (abbreviated as CCP) introduced by P.L. Lions \cite{Lions85} has played a crucial role in demonstrating the precompactness of Palais-Smale sequences. This CCP has been extended to Sobolev spaces with variable exponents, allowing for the exploration of critical problems involving variable exponents, see e.g.,  \cite{BS10,CH21,Fu09,FZ10,HK21, HKL22} and references therein. 
 In order to address our critical problem, we establish a novel concentration-compactness principle for the Musielak-Orlicz-Sobolev space $W_0^{1,\mathcal{H}}(\Omega)$ (the completion of $C^\infty_c(\Omega)$ in $\Wp{\mathcal{H}}$), which stands as an independent contribution in our research. It is noteworthy that our CCP is a direct extension of the one obtained in \cite{BS10} when considering the specific case of $a(\cdot)\equiv 0$. 

Upon acquiring the CCP, we employ it to establish several multiplicity results for problem \eqref{e1.1}. One can observe that \eqref{e1.1} involves the presence of a Kirchhoff term, which needs to be treated more carefully. To address this, we introduce a modified problem by truncating the Kirchhoff term, and subsequently showed that the solutions of the modified problem are also solutions for \eqref{e1.1}. Another highlight of our work is that we impose local conditions on the reaction term. Specifically, the right-hand side of \eqref{e1.1} exhibits concave-convex/ superlinear behavior only within a subset of $\Omega$. By employing techniques from calculus and variations, such as the genus theory and deformation lemma, we are able to derive multiplicity results for the generalized concave-convex/ superlinear problems. Remarkably, our result is the first attempt for double problems with growth that surpasses the threshold of $p^*(\cdot)$ to the authors’ best knowledge.

We organize the paper as follows. In Section~\ref{MainResults}, we present our main results: the CCP for $W_0^{1,\mathcal{H}}(\Omega)$ (Theorem~\ref{Theo.ccp}), the existence of infinitely many solutions for the generalized concave-convex problem (Theorem~\ref{Theo.cc}), and the existence of many solutions for the generalized superlinear problem (Theorems~\ref{Theo.sl.M=1} and \ref{Theo2.sl.M=1}). The theoretical framework for our analysis is introduced in Section~\ref{Pre}. The proof of the CCP is provided in Section~\ref{CCP}, while the multiplicity results are addressed in Section~\ref{Existence}.

  \section{Main Results}\label{MainResults}
  
 Throughout the paper we denote 
 $$C^{0,1}(\close):=\{h:\overline{\Omega} \to \mathbb{R},\ h\ \text{ is Lipschitz continuous}\},$$
  $$C_+(\overline{\Omega}):=\{h\in C(\overline{\Omega}): \ \min_{x\in \overline{\Omega}}h(x)>1\},$$
  and for $h\in C(\overline{\Omega})$,
  $$h^+:=\max_{x\in \overline{\Omega}}h(x),\ \ h^-:=\min_{x\in \overline{\Omega}}h(x).$$
The notations $u_n\to u$ (resp. $u_n \rightharpoonup u,u_n \overset{\ast }{\rightharpoonup }u$) stand for the term $u_n$ strongly (resp. weakly, weakly-$\ast$) converges to $u$ as $n \to \infty$ in an appropriate normed space (from the context), and $\langle \cdot,\cdot\rangle$ stands for the duality pairing between that normed space and its dual. An open ball in $\R^N$ centered at $x_0$ with radius $\epsilon$ is denoted by $B_\epsilon(x_0)$. The notations $C_i$ ($i\in\N$) always stand for a positive constant given from the assumptions or dependent only on the data.
 
 In the following, the functions $\mathcal{H}$ and $\mathcal{B}$ are respectively defined as
 \begin{equation}\label{def_H}
 	\mathcal{H}(x,t):=t^{p(x)} +a(x)t^{q(x)}\ \text{for} \ (x,t)\in \overline{\Omega}\times [0,\infty)
 \end{equation}
 and
 \begin{equation}\label{def_B}\mathcal{B}(x,t):=c_1(x)|t|^{r_1(x)}+c_2(x)a(x)^{\frac{r_2(x)}{q(x)}}|t|^{r_2(x)}\ \text{for} \ (x,t)\in \overline{\Omega}\times \R,
 \end{equation}
 where the involved functions satisfy the following assumptions:
 \begin{itemize}
 	\item [$(\mathcal{H}_A)$] $a,p,q \in  C^{0,1}(\close)$ such that $a(x) \geq 0,\, 1<p(x)<q(x)<N$ and $\displaystyle \frac{q(x)}{p(x)}<1+\frac{1}{N}$  for all $x \in \close$.
 	\item [($\mathcal{H}_B$)] $c_1,c_2\in L^\infty(\Omega)$ and $r_1,r_2 \in C_+(\overline{\Omega})$ such that $c_1(x)>0$, $c_2(x)\geq 0$, $q(x)<r_1(x)\leq p^\ast(x)$,  $p^\ast(x)-r_1(x)=q^\ast(x)-r_2(x)$ for all $x \in \close$ and
 	$$\mathcal{C}:=\{x \in \overline{\Omega}:\, r_1(x)=p^\ast(x)\ \text{and}\ r_2(x)=q^\ast(x)\} \neq \emptyset.$$
 \end{itemize}
 Under these assumptions, it holds that $\Wpzero{\mathcal{H}}\hookrightarrow L^{\mathcal{B}}(\Omega)$; hence,
 \begin{equation}\label{S}
 	S:=\underset{\phi\in \Wpzero{\mathcal{H}}\setminus\{0\}}{\inf}
 	\frac{\|\phi\|}{\| \phi \|_{\mathcal{B}}}>0,
 \end{equation}
see Section~\ref{Pre} for the definitions and properties of the Musielak-Orlicz space $\left(L^{\mathcal{B}}(\Omega),\|\cdot\|_{\mathcal{B}}\right)$ and the Musielak-Orlicz-Sobolev space $\left(\Wpzero{\mathcal{H}},\|\cdot\|\right)$. 
 \subsection{The concentration-compactness principle for $\Wpzero{\mathcal{H}}$}${}$

Let $\mathcal{M}(\overline{\Omega})$ denote the space of Radon measures on $\overline{\Omega}$, namely, the dual space of $C(\overline{\Omega})$. By the Riesz representation theorem, for each $\mu\in \mathcal{M}(\overline{\Omega}),$ there is a unique signed Borel measure on $\overline{\Omega}$, still denoted by $\mu$ itself, such that
$$\langle \mu,f\rangle=\int_{\overline{\Omega}}f\diff\mu,\quad \forall f\in C(\overline{\Omega}).$$
The space $L^1(\Omega)$ is identified with a subspace of $\mathcal{M}(\overline{\Omega})$ through the mapping
$T:\ L^1(\Omega)\to \mathcal{M}(\overline{\Omega})$  defined as 
$$\langle Tu,f\rangle=\int_{\Omega}uf\diff x \ \ \forall u\in L^1(\Omega), \ \forall f\in C(\overline{\Omega})$$
(see, e.g., \cite[p. 116]{Bre.Book}).

	\vskip5pt
	The following theorem is a Lions type concentration-compactness principle for the Musielak-Orlicz-Sobolev space $\Wpzero{\mathcal{H}}$. 

	%======================STATEMENT OF THE CCP=========================%
\begin{theorem}\label{Theo.ccp} %{\rm \textbf{(The concentration-compactness principle for $\Wpzero{\mathcal{H}}$)}} 
	Let $(\mathcal{H}_A)$ and $(\mathcal{H}_B)$ hold.  Let $\{u_n\}_{n\in\mathbb{N}}$ be a bounded sequence in $\Wpzero{\mathcal{H}}$ such that
           \begin{gather}
			u_n \rightharpoonup u \quad \text{in}\quad  \Wpzero{\mathcal{H}}, \label{ccp.weak}\\
			\mathcal{H}(\cdot,|\nabla u_n|) \overset{\ast }{\rightharpoonup } \mu\quad \text{in}\quad \mathcal{M}(\overline{\Omega}), \label{ccp.mu}\\
			\mathcal{B}(\cdot,u_n)\overset{\ast }{\rightharpoonup }\nu\quad \text{in}\quad \mathcal{M}(\overline{\Omega}). \label{ccp.nu}
		\end{gather}
		Then, there exist $\{x_i\}_{i\in I}\subset \mathcal{C}$ of distinct points and $\{\nu_i\}_{i\in I}, \{\mu_i\}_{i\in I}\subset (0,\infty),$ where $I$ is at most countable, such that
		\begin{gather}
			\nu=\mathcal{B}(\cdot,u) + \sum_{i\in I}\nu_i\delta_{x_i},\label{T.ccp.form.nu}\\
			\mu \geq \mathcal{H}(\cdot,|\nabla u|) + \sum_{i\in I} \mu_i \delta_{x_i},\label{T.ccp.form.mu}\\
			S \min \left\{\nu_i^{\frac{1}{p^\ast(x_i)}},\nu_i^{\frac{1}{q^\ast(x_i)}}  \right\}\leq \max \left\{\mu_i^{\frac{1}{p(x_i)}},\mu_i^{\frac{1}{q(x_i)}}\right\}, \quad \forall i\in I,\label{T.ccp.nu_mu}
		\end{gather}
		where $\delta_{x_i}$ is the Dirac mass at $x_i$ and $S$ is given by \eqref{S}.
		
	\end{theorem}
\begin{remark}\rm It is worth pointing out that Theorem~\ref{Theo.ccp} is a direct extension of \cite[Theorem 1.1]{BS10} by taking $a(\cdot)\equiv 0$ and $c_1(\cdot)\equiv 1$.
\end{remark}

\subsection{The multiplicity of solutions}${}$

By employing Theorem~\ref{Theo.ccp}, we investigate the  multiplicity of solutions for the following problem
	\begin{equation}\label{e1.1'}
	\begin{cases}
		- M\left(\int_\Omega \mathcal{A}(x,\nabla u)\diff x\right)\operatorname{div}A(x,\nabla u) = \lambda f(x,u) +\theta B(x,u)& ~\text{in}~ \Omega ,\\
		u=0& ~\text{on} ~ \partial \Omega 
	\end{cases}
\end{equation}
under the assumptions ($\mathcal{H}_A$) and ($\mathcal{H}_B$), where $B$ is a critical term of the form
$$B(x,t):=c_1(x)|t|^{r_1(x)-2}t+c_2(x)a(x)^{\frac{r_2(x)}{q(x)}}|t|^{r_2(x)-2}t,\quad (x,t) \in \overline{\Omega}\times \mathbb{R}.$$ Moreover, the Kirchhoff term $M$ and the nonlinear term $f$ satisfy the following assumptions.
\begin{itemize}
	\item[$(\mathcal{M})$] $M:[0,+\infty) \to \R$ is a real function such that $M$ is continuous and non-decreasing on some interval $[0,\tau_0)$ and $m_0:=M(0)>0$.
		\item[$(\mathcal{F}_0)$] There exists a function $\alpha \in C_+(\overline\Omega)$ such that $q^+<\alpha^-\leq \alpha(x) < r_1(x)$ for all $x\in \overline\Omega$ and 
	$$
	|f(x,t)| \leq C_1\left(1+|t|^{\alpha(x)-1}\right) \ \ \text{for a.a.} \ x\in\Omega\ \text{and all} \ t\in \R.$$
	\item[$(\mathcal{F}_1)$] 	$f(x,-t)=-f(x,t)$  for a.a. $x\in\Omega$ and all $t\in\R$.
\end{itemize}
By a solution $u$ of  problem \eqref{e1.1'} we mean $u\in W_0^{1,\mathcal{H}}(\Omega)$ such that
\begin{multline*}
	M\left(\int_\Omega \mathcal{A}(x,\nabla u)\diff x\right)\int_\Omega A(x,\nabla u)\cdot\nabla v\,\diff x-\lambda\int_\Omega f(x,u)v\,\diff x\\-\theta\int_\Omega B(x,u)v\,\diff x=0,\quad \forall v\in \Wpzero{\mathcal{H}}.
\end{multline*}
This definition is well defined in view of the embedding results on $\Wpzero{\mathcal{H}}$ that are provided in Section~\ref{Pre}. In the following, we denote  
$$F(x,t) := \int_0^tf(x,s)\diff s\ \ \text{for a.a.} \ x\in \Omega \ \text{and all} \ t\in\R.$$
Clearly, $(\mathcal{F}_0)$ implies that
\begin{equation}\label{q-r1}
	q^+<r_1^-
\end{equation}
and 
\begin{equation}\label{F}
	|F(x,t)| \leq 2C_1\left(1+|t|^{\alpha(x)}\right) \ \ \text{for a.a.} \ x\in\Omega\ \text{and all} \ t\in \R.
\end{equation}
In most our multiplicity results, we will need the following condition:
\begin{itemize}
	\item[$(\mathcal P)$] $\kappa_1:=\inf_{\varphi\in C_c^\infty(\Omega)\setminus\{0\}}\frac{\int_\Omega|\nabla\varphi|^{p(x)}\diff x}{\int_\Omega|\varphi|^{p(x)}\diff x}$ \ is positive.
\end{itemize}
 Obviously, $(\mathcal P)$ automatically holds when $p(\cdot)$ is constant. When $p(\cdot)$ is non-constant, $(\mathcal P)$ also holds if there exists a vector $\ell\in\mathbb{R}^N\setminus\{0\}$ such that for any $x\in\Omega$, $\eta(t):=p(x + t\ell)$ is monotone for $t \in I_x:= \{t\in\R: \, x + t\ell\in\Omega\}$. Note that $(\mathcal P)$ does not hold if there is an open ball $B_{\epsilon}(x_0)\subset\Omega$ such that $p(x_0)<$ (or $>$) $p(x)$ for all $x\in\partial B_{\epsilon}(x_0)$, see \cite[Theorems 3.1 and 3.3]{FZZ05}. From the definitions of $\kappa_1$ and $\Wpzero{\mathcal{H}}$, it holds
\begin{equation}\label{mu1}
	\kappa_1 \int_\Omega|\varphi|^{p(x)}\diff x\leq \int_\Omega|\nabla\varphi|^{p(x)}\diff x,\quad \forall \varphi\in \Wpzero{\mathcal{H}}.
\end{equation}

Next, we are going to state our multiplicity results for \eqref{e1.1'} by considering the right-hand side of the generalized concave-convex/ superlinear type.
	
\subsubsection{\textbf{Problems of a generalized concave-convex type}}${}$

We first investigate the multiplicity of solutions to problem \eqref{e1.1'} when $\theta=1$ and the right-hand side is of a generalized concave-convex type. More precisely, in addition to $(\mathcal{F}_0)$ and $(\mathcal{F}_1)$, we further assume the following assumption.
	\begin{itemize}

		\item[$(\mathcal{F}_2)$] There exist a function $\sigma\in C_+(\overline\Omega)$ with $\sigma^+<p^-$ and a ball $B\subset\Omega$ such that
		\begin{itemize}
			\item [(i)]  $0\leq F(x,t)\leq C_2|t|^{\sigma(x)}+C_3|t|^{p(x)}$ and $r_1^- F(x,t)-f(x,t)t\leq C_4|t|^{\sigma(x)}+C_5|t|^{p(x)}$ for a.a. $x\in\Omega$ and all $t\in\R$;
			\item [(ii)]  $C_6|t|^{\sigma(x)}\leq F(x,t)$ for a.a. $x\in B$ and all $t\in\R$.
		\end{itemize} 
	\end{itemize}
 We have the existence of infinitely many solutions for generalized concave-convex problems as follows.
 	\begin{theorem}\label{Theo.cc} %{\rm \textbf{(Infinitely many solutions for the generalized concave-convex type problem)}}
	Let  $(\mathcal{H}_A)$, $(\mathcal{H}_B)$, $(\mathcal P)$,  $(\mathcal{M})$, $(\mathcal{F}_0)$, $(\mathcal{F}_1)$ and $(\mathcal{F}_2)$ hold. Then, there exists $\lambda_\ast>0$ such that for any $\lambda\in(0,\lambda_\ast)$, problem \eqref{e1.1'} with $\theta=1$ admits infinitely many solutions. Furthermore, if  $u_\lambda$ is a solution corresponding to $\lambda$, then it holds that 
		$$\lim_{\lambda \to 0^+}\|u_\lambda\| = 0.$$
	\end{theorem}
\begin{example}\rm
Some examples of $f$ fulfilling $(\mathcal{F}_0)-(\mathcal{F}_2)$ are
\begin{itemize}
	\item [(i)] $f(x,t)=c_1|t|^{\delta(x)-2}t+c_2|t|^{p(x)-2}t$,
	\item [(ii)]$f(x,t)= c_1|t|^{\delta(x)-2}t\log^{\kappa(x)} (e+|t|)+c_2|t|^{m(x)-2}t$,
	\item [(iii)]$f(x,t)= c_1|t|^{n(x)-2}t$,
\end{itemize} 
where $\delta,\kappa,m\in C_+(\overline\Omega)$ satisfy $\delta^+<p^-$, $\delta(\cdot)\leq m(\cdot)\leq p(\cdot)$, $n=\varphi\delta+(1-\varphi)m$ with $\varphi\in C(\overline\Omega)$, $0\leq\varphi(\cdot)\leq 1$, $\operatorname{supp}\varphi\subset B_{2\epsilon_0}(x_0)\subset\Omega$ and $\varphi(\cdot)\equiv 1$ on $B_{\epsilon_0}(x_0)$; and $c_1>0$, $c_2\geq 0$ are constants. Note that the case (iii) gives an example that $f$ can only be $p(\cdot)$-sublinear locally.
\end{example}

\subsubsection{\textbf{Problems of a generalized superlinear type}}${}$

Finally, we investigate the multiplicity of solutions to problem \eqref{e1.1'} when $M(\cdot) \equiv 1$, $\lambda=1$ and the right-hand side is of a generalized superlinear type. More precisely, in addition to $(\mathcal{F}_0)$ and $(\mathcal{F}_1)$, we further assume the following assumptions.
\begin{itemize}
	\item [$(\mathcal{F}_3)$] With $\kappa_1$ given by $(\mathcal P)$, there exist a constant $\beta \in [q^+,r_1^-)$ and a function $e \in L^1(\Omega)$ such that
	$$\beta F(x,t)-f(x,t)t \leq \frac{\left(\beta-q^+\right)\kappa_1}{q^+}|t|^{p(x)}+e(x),\quad  \text{for a.a.} \ x\in\Omega\ \text{and all} \ t\in \R.$$
	\item [$(\mathcal{F}_4)$] There is a ball $B \subset \Omega$ such that
	$$\lim_{|t| \to +\infty}\frac{F(x,t)}{|t|^{q_B^+}}=+\infty,\quad\text{ uniformly for a.a. } x\in B,$$
	where $q^+_B=\max\limits_{x\in \overline{B}} q(x)$.
\end{itemize}
The next theorem is our multiplicity result for the generalized superlinear case.
\begin{theorem}\label{Theo.sl.M=1}
	Let $(\mathcal{H}_A), (\mathcal{H}_B)$, $(\mathcal P)$, $(\mathcal{F}_0)$, $(\mathcal{F}_1)$,  $(\mathcal{F}_3)$ and $(\mathcal{F}_4)$ hold with $c_1^{-\frac{\alpha}{r_1-\alpha}}\in L^1(\Omega)$. Then, for each $n\in \mathbb{N}$, there exists $\theta_n>0$ such that for any $\theta \in (0,\theta_n)$, problem \eqref{e1.1'} with $M(\cdot) \equiv 1$ and $\lambda=1$ possesses at least $n$ pairs of nontrivial solutions.
\end{theorem}
In our last main result below, we can drop $(\mathcal P)$ in Theorem~\ref{Theo.sl.M=1} if $(\mathcal{F}_3)$ is replaced with a stronger assumption.
\begin{theorem}\label{Theo2.sl.M=1}
	Let $(\mathcal{H}_A), (\mathcal{H}_B)$, $(\mathcal{F}_0)$, $(\mathcal{F}_1)$ and $(\mathcal{F}_4)$ hold with $c_1^{-\frac{\alpha}{r_1-\alpha}}\in L^1(\Omega)$ and assume in addition that 
	\begin{itemize}
		\item [$(\mathcal{F}_5)$] There exist a constant $\beta \in [q^+,r_1^-)$ and a function $e \in L^1(\Omega)$ such that
		$$\beta F(x,t)-f(x,t)t \leq e(x),\quad \text{for a.a.} \ x\in\Omega\ \text{and all} \ t\in \R.$$
	\end{itemize}
Then, the conclusion of Theorem~\ref{Theo.sl.M=1} remains valid.
\end{theorem}

\begin{example}\rm
	Some examples of $f$ fulfilling $(\mathcal{F}_0)$, $(\mathcal{F}_1)$, $(\mathcal{F}_3)$ and $(\mathcal{F}_4)$ (with $\beta=q^+$) are
	\begin{itemize}
		\item [(i)] $f(x,t)=c_1|t|^{\delta(x)-2}t+c_2|t|^{q^+-2}t$,
		\item [(ii)]$f(x,t)= c_1|t|^{\delta(x)-2}t+c_2|t|^{m(x)-2}t$,
		\item [(iii)]$f(x,t)= c_1|t|^{n(x)-2}t$,
			\end{itemize} 
	where $\delta,m\in C_+(\overline\Omega)$ satisfy $q^+\leq \delta^-$, $q^+<\delta(x_0)$ at some $x_0\in\Omega$, $m(\cdot)< p(\cdot)$, $n=\varphi \delta+(1-\varphi)q^+$ with $\varphi\in C(\overline\Omega)$, $0\leq\varphi(\cdot)\leq 1$, $\operatorname{supp}\varphi\subset B_{2\epsilon_0}(x_0)\subset\Omega$ and $\varphi(\cdot)\equiv 1$ on $B_{\epsilon_0}(x_0)$; and $c_1>0$, $c_2\geq 0$ are constants. Moreover, the functions $f$ in (i) and (ii) also satisfy $(\mathcal{F}_5)$. It is clear that the right-hand side of \eqref{e1.1'} corresponding these examples can only be $q^+$-superlinear locally.
\end{example}
		
	%=======================SECTION 2. PRELIMINARIES====================== %
	\section{Variable exponent spaces}\label{Pre}
	%------------------------------------------------------------------------------------------------------------------------------------
	In this section, we briefly review variable exponent spaces and refer the reader to \cite{CGHW, DHHR,HW22} for a systematic study on these spaces. Let $\Omega$ be a bounded domain in $\mathbb{R}^N$ with Lipschitz boundary.
	\subsection{Lebesgue spaces with variable exponents}${}$
	
\vskip5pt

  For  $m\in C_+(\overline\Omega)$ and a $\sigma$-finite, complete measure $\mu$ in $\overline{\Omega},$  define the  variable exponent Lebesgue space $L_\mu^{m(\cdot)}(\Omega)$ as
	$$
	L_\mu^{m(\cdot)}(\Omega) := \left \{ u : \Omega\to\mathbb{R}\  \hbox{is}\  \mu-\text{measurable},\ \int_\Omega |u(x)|^{m(x)} \;\diff\mu < \infty \right \},
	$$
	endowed with the Luxemburg norm
	$$
	\|u\|_{L_\mu^{m(\cdot)}(\Omega)}:=\inf\left\{\lambda >0:
	\int_\Omega
	\Big|\frac{u(x)}{\lambda}\Big|^{m(x)}\;\diff\mu\le1\right\}.
	$$
For a Lebesgue measurable and positive a.e. function $w$, denote $L^{m(\cdot)}(w,\Omega):=L_\mu^{m(\cdot)}(\Omega)$ with $\diff \mu=w(x)\diff x$. When $\mu$ is the Lebesgue measure, we write  $\diff x$, $L^{m(\cdot) }(\Omega) $  and $\|\cdot\|_{m(\cdot)}$ in place of $\diff \mu$, $L_\mu^{m(\cdot)}(\Omega)$ and $\|\cdot\|_{L_\mu^{m(\cdot)}(\Omega)}$, respectively. %Generally, the notation $\|\cdot\|_E$ stands for a norm on the normed space $E$.
 
	The following propositions are crucial for our arguments in the next sections.
	
	%---------------PROP1: HOLDER INEQUALITY-------------------------%
	\begin{proposition}[\cite{DHHR}] \label{prop.Holder}
		The space $L_{\mu}^{m(\cdot) }(\Omega )$ is a separable and uniformly convex Banach space, and its conjugate space is $L_{\mu}^{m'(\cdot) }(\Omega ),$ where  $1/m(x)+1/m'(x)=1$. For any $u\in L_{\mu}^{m(\cdot)}(\Omega)$ and $v\in L_{\mu}^{m'(\cdot)}(\Omega)$, we have
		\begin{equation*}
			\left|\int_\Omega uv\,\diff \mu\right|\leq\ 2 \|u\|_{L_\mu^{m(\cdot)}(\Omega)}\|v\|_{L_\mu^{m'(\cdot)}(\Omega)}.
		\end{equation*}
	\end{proposition}
	Define the modular $\rho :L_{\mu}^{m(\cdot) }(\Omega )$ $ \to \mathbb{R}$ as
	\[
	\rho (u) :=\int_{\Omega }| u(x)| ^{m(x) }\diff \mu,\quad
	\forall u\in L_{\mu}^{m(\cdot) }(\Omega ) .
	\]
	
	%---------------PROP2: ESTIMATING MODULAR BY NORM--------------------%
\begin{proposition}[\cite{DHHR}] \label{prop.nor-mod.M}
For all $u\in L_\mu^{m(\cdot) }(\Omega ),$  we have
		\begin{itemize}
			\item[(i)] $\|u\|_{L_{\mu}^{m(\cdot)}(\Omega)}<1$ $(=1,>1)$
			if and only if \  $\rho (u) <1$ $(=1,>1)$, respectively;
			
			\item[(ii)] if \  $\|u\|_{L_{\mu}^{m(\cdot)}(\Omega)}>1,$ then  $\|u\|^{m^{-}}_{L_{\mu}^{m(\cdot)}(\Omega)}\leq \rho (u) \leq \|u\|_{L_{\mu}^{m(\cdot)}(\Omega)}^{m^{+}}$;
			\item[(iii)] if \ $\|u\|_{L_{\mu}^{m(\cdot)}(\Omega)}<1,$ then $\|u\|_{L_{\mu}^{m(\cdot)}(\Omega)}^{m^{+}}\leq \rho
			(u) \leq \|u\|_{L_{\mu}^{m(\cdot)}(\Omega)}^{m^{-}}$.
		\end{itemize}
Consequently,
		$$\|u\|_{L_{\mu}^{m(\cdot)}(\Omega)}^{m^{-}}-1\leq \rho (u) \leq \|u\|_{L_{\mu}^{m(\cdot)}(\Omega)}^{m^{+}}+1,\ \forall u\in L_\mu^{m(\cdot)}(\Omega ).$$
	\end{proposition}

\subsection{A class of Musielak-Orlicz-Sobolev spaces} \label{A gene.Orlicz space}${}$
\vskip5pt

Define $\Phi: \overline{\Omega}\times \R\to\R$ as
\begin{align}\label{Phi}
	\Phi(x,t):=b(x)t^{r(x)}+c(x)t^{s(x)}
	\quad\text{for } (x,t)\in \overline{\Omega}\times [0,\infty),
\end{align}
where $r,s\in C_+(\close)$ with $r(\cdot)<s(\cdot)$, $0< b(\cdot) \in \Lp{1}$ and $0 \leq c(\cdot) \in \Lp{1}$.
The modular associated with $\Phi$ is defined as
\begin{align}\label{modular-Lp}
	\rho_{\Phi}(u):= \into \Phi (x,|u|)\,\diff x.
\end{align}
Then, the corresponding Musielak-Orlicz space $\Lp{\Phi}$ is defined  
by
\begin{align*}
	L^{\Phi}(\Omega):=\left \{u : \Omega\to\mathbb{R}\  \hbox{is Lebesgue measurable},\,\rho_{\Phi}(u) < +\infty \right\},
\end{align*}
endowed with the norm
\begin{align*}
	\|u\|_{\Phi}:= \inf \left \{ \tau >0 :\, \rho_{\Phi}\left(\frac{u}{\tau}\right) \leq 1  \right \}.
\end{align*}
In view of \cite[Proposition 2.6]{CGHW},  $\Lp{\Phi}$ is a Banach space. The following proposition gives the relation between the modular $\rho_{\Phi}$ and its norm $\|\cdot\|_{\Phi}$, see \cite[Proof of Proposition 2.13]{CGHW} for a detail proof.

\begin{proposition}\label{prop.nor-mod.D}
	Let $u, u_n\in\Lp{\Phi}$ ($n\in\N$) and let $\rho_{\Phi}$ be defined as in \eqref{modular-Lp}.
	\begin{enumerate}
		\item[\textnormal{(i)}]
		If $u\neq 0$, then $\|u\|_{\Phi}=\lambda$ if and only if $ \rho_{\Phi}(\frac{u}{\lambda})=1$.
		\item[\textnormal{(ii)}]
		$\|u\|_{\Phi}<1$ (resp.\,$>1$, $=1$) if and only if $ \rho_{\Phi}(u)<1$ (resp.\,$>1$, $=1$).
		\item[\textnormal{(iii)}]
		If $\|u\|_{\Phi}<1$, then $\|u\|_{\Phi}^{s^+}\leqslant \rho_{\Phi}(u)\leqslant\|u\|_{\Phi}^{r^-}$.
		\item[\textnormal{(iv)}]
		If $\|u\|_{\Phi}>1$, then $\|u\|_{\Phi}^{r^-}\leqslant \rho_{\Phi}(u)\leqslant\|u\|_{\Phi}^{s^+}$.
		\item[\textnormal{(v)}] $\|u_n\|_{\Phi}\to 0$ as $n\to \infty$ if and only if $\rho_{\Phi}(u_n)\to 0$ as $n\to \infty$.
	\end{enumerate}
\end{proposition}

Let $\mathcal{H}$ and $\mathcal{B}$ be as in \eqref{def_H} and \eqref{def_B}, respectively and in the following, we always assume that ($\mathcal{H}_A$) and ($\mathcal{H}_B$) are fulfilled. We define the Musielak-Orlicz-Sobolev space $\Wp{\mathcal{H}}$ as
\begin{align*}
	\Wp{\mathcal{H}}
	=\left \{u \in L^{\mathcal{H}}(\Omega) \,:\,|\nabla u| \in L^{\mathcal{H}}(\Omega) \right \}
\end{align*}
equipped with the norm
\begin{align*}
	\|u\|_{1,\mathcal{H}} = \|u\|_{\mathcal{H}}+\|\nabla u\|_{\mathcal{H}},
\end{align*}
where $\|\nabla u\|_{\mathcal{H}}=\| \, |\nabla u| \,\|_{\mathcal{H}}$. We denote by $\Wpzero{\mathcal{H}}$ the completion of $C^\infty_c(\Omega)$ in $\Wp{\mathcal{H}}$. In view of \cite[Proposition 2.12]{CGHW}, $\Lp{\mathcal{H}}$, $\Wp{\mathcal{H}}$ and $\Wpzero{\mathcal{H}}$ are reflexive Banach spaces. When $a(\cdot)\equiv 0$, we write $W^{1,p(\cdot)}(\Omega)$, $W_0^{1,p(\cdot)}(\Omega)$ and $\|\cdot\|_{1,p(\cdot)}$ in place of $\Wp{\mathcal{H}}$, $\Wpzero{\mathcal{H}}$ and $\|\cdot\|_{1,\mathcal{H}}$, respectively. The following embedding results can be found in \cite[Proposition 2.16]{CGHW}, recalling 
$$m^*(\cdot):=\frac{Nm(\cdot)}{N-m(\cdot)}$$
for $m(\cdot)<N$.

\begin{proposition}\label{prop_embs}
	The following assertions hold.
	\begin{enumerate}
				\item[\textnormal{(i)}] $\Wp{\mathcal{H}} \hookrightarrow \Lp{r(\cdot)}$ and $\Wpzero{\mathcal{H}} \hookrightarrow \Lp{r(\cdot)}$ for $r \in C(\close)$ with $ 1 \leq r(x) \leq p^*(x)$ for all $x\in \close$;
		\item[\textnormal{(ii)}]
		$\Wp{\mathcal{H}} \hookrightarrow \hookrightarrow\Lp{r(\cdot)}$ and $\Wpzero{\mathcal{H}} \hookrightarrow \hookrightarrow\Lp{r(\cdot)}$  for $r \in C(\close) $ with $ 1 \leq r(x) < p^*(x)$ for all $x\in \close$.
	\end{enumerate}
\end{proposition}
	
The following Poincar\'e-type inequality was proved in \cite{CGHW}.
\begin{proposition}\label{prop_Poincare}
	The following Poincar\'e-type inequality holds:
			\begin{align*}
			\|u\|_{\mathcal{H}} \leq C\|\nabla u\|_{\mathcal{H}}\quad\text{for all 
			} u \in \Wpzero{\mathcal{H}}.
		\end{align*}
	Consequently, on $\Wpzero{\mathcal{H}}$ we have an equivalent norm 
	\begin{equation}\label{norm}
		\|\cdot\|:=\|\nabla \cdot\|_{\mathcal{H}}.
	\end{equation}
\end{proposition}
The next proposition is \cite[Proposition 3.7]{HW22}, and it plays a key role in obtaining our main results.
\begin{proposition}\label{prop_S-C-E}
	Define 
	\begin{align*}
		\Psi(x,t):=|t|^{r(x)}+a(x)^{\frac{s(x)}{q(x)}}|t|^{s(x)}
		\quad\text{for } (x,t)\in \overline{\Omega}\times \R,
	\end{align*}
	where $r,s\in C_+(\close)$ satisfy $r(x)\leq p^*(x)$ and $s(x)\leq q^*(x)$ for all $x\in\overline{\Omega}$. Then, we have the continuous embedding
	\begin{align}\label{Prop-S-E}
		W^{1,\mathcal{H}}(\Omega)
		\hookrightarrow  L^{\Psi}(\Omega).
	\end{align}
	Furthermore, if $r(x)<p^*(x)$ and $s(x)< q^*(x)$ for all $x\in\overline{\Omega}$, then the embedding in \eqref{Prop-S-E} is compact. In particular, it holds
	\begin{align}\label{Prop-S-C}
		W^{1,\mathcal{H}}(\Omega)
		\hookrightarrow \hookrightarrow L^{\mathcal{H}}(\Omega).
	\end{align}
\end{proposition}
In the next sections, we will apply Proposition~\ref{prop_S-C-E} with $\Psi=\mathcal{B}$, where  $\mathcal{B}$ is given by \eqref{def_B}. This application is valid since $c_1,c_2\in L^\infty(\Omega)$. As demonstrated in \cite[Proposition 3.5]{HW22}, for the case of constant exponents, $\mathcal{G}^*(x,t):=|t|^{p^*}+a(x)^{\frac{q^*}{q}}|t|^{q^*}$ is the optimal one among $\mathcal{B}_{r,s,\alpha}$ of the form
\begin{align*}%\label{B_alpha}
	\mathcal{B}_{r,s,\alpha}(x,t):=|t|^{r}+a(x)^{\alpha} |t|^{s}
	\quad \text{for }(x,t)\in \overline{\Omega}\times \R,
\end{align*}
where $r,s\in (1,\infty)$ and $\alpha\in (0,\infty)$, such that the following continuous embedding holds
\begin{align}\label{embedding-critical-domain-O}
	W^{1,\mathcal{H}}(\Omega)\hookrightarrow L^{\mathcal{B}_{r,s,\alpha}}(\Omega),
\end{align}
namely, if \eqref{embedding-critical-domain-O} holds for any data $(p,q,a,\Omega)$ satisfying the assumption $(\mathcal{H}_A)$, then there must be $r\leq p^*$, $s\leq q^*$ and $\alpha\geq \frac{q^*}{q}$.

Finally, in the next sections we frequently use Young's inequality of the form
\begin{equation}\label{young}
	ab\leq \frac{1}{m(x)}\eps a^{m(x)}+\frac{m(x)-1}{m(x)}\eps^{-\frac{1}{m(x)-1}}b^{\frac{m(x)}{m(x)-1}}\leq \eps a^{m(x)}+\left(1+\eps^{-\frac{1}{m^--1}}\right)b^{\frac{m(x)}{m(x)-1}} 
\end{equation}
 for all $a,b\geq 0$,  $\eps>0$, $x\in \overline{\Omega}$,  and $m\in C_+(\close)$.
%======================SECTION 3. THE CCP======================== %
\section{Proof of The Concentration-Compactness Principle}\label{CCP}
In this section we will prove Theorem~\ref{Theo.ccp} by modifying the idea used in \cite {BS10} that extended the concentration-compactness principle by P.L. Lions \cite{Lions85} to the variable exponent case, namely, a concentration-compactness principle for $W_0^{1,p(\cdot)}(\Omega)$. Unlike this case, dealing with the challenges posed by the double phase operator requires significantly more intricate arguments.

Before giving a proof of Theorem~\ref{Theo.ccp}, we recall some auxiliary results obtained in \cite{BS10}.

%====================AUXILIARY LEMMAS========================%
\begin{lemma}[\cite{BS10}] \label{L.convergence}
Let $\nu,\{\nu_n\}_{n\in\mathbb{N}}$ be nonnegative and finite Radon measures on $\overline{\Omega}$ such that $\nu_n\overset{\ast }{\rightharpoonup } \nu$ in $\mathcal{M}(\overline{\Omega})$. Then, for any $m\in C_{+}(\overline{\Omega })$,
$$
\|\phi\|_{L^{m(\cdot)}_{\nu_n}(\overline{\Omega})} \to
\|\phi\|_{L^{m(\cdot)}_{\nu}(\overline{\Omega})}, \quad \forall \phi\in C(\overline{\Omega}).$$	
\end{lemma}	
	
\begin{lemma}[\cite{BS10}]\label{L.reserveHolder}
Let $\mu,\nu$ be two nonnegative and finite Borel measures
on $\overline{\Omega}$, such that there exists some  constant $C>0$ holding
$$
\|\phi\|_{L_\nu^{t(\cdot)}(\overline{\Omega})}\leq C\|\phi\|_{L_\mu^{s(\cdot)}(\overline{\Omega})},\ \ \forall \phi\in C^\infty(\overline{\Omega})
$$
for some $s,t\in C_+(\overline{\Omega})$ satisfying $s(x)<t(x)$ for all $x\in \overline{\Omega}$. Then, there exist an at most countable set $\{x_i\}_{i\in I}$ of distinct points in $\overline{\Omega}$ and
$\{\nu_i\}_{i\in I}\subset (0,\infty)$, such that
$$
\nu=\sum_{i\in I}\nu_i\delta_{x_i}.
$$	
\end{lemma}	
	
The subsequent result is an extension of the Brezis-Lieb Lemma to the Musielak-Orlicz spaces $L^{\Phi}(\Omega)$. The proof is a direct consequence of \cite[Lemma 3.6]{HS16}, thus we omit it.
\begin{lemma}\label{L.brezis-lieb}
	Let $\Phi$ be as in \eqref{Phi}. Let $\{f_n\}_{n\in\N}$ be a bounded sequence in $L^{\Phi}(\Omega)$
 and $f_n(x)\to f(x)$ a.a. $x\in\Omega$. Then $f\in L^{\Phi}(\Omega)$ and
$$
\lim_{n\to\infty}\int_\Omega \left|\Phi(x,|f_n|)
-\Phi(x,|f_n-f|)-\Phi(x,|f|)\right|\diff x=0.
$$	
%Consequently, thanks to the Dominated Convergence Theorem, for any $\phi \in C(\overline{\Omega})$, we have
%$$ \lim_{n\to\infty}\int_\Omega \left|\phi\mathcal{T}(x,|f_n|)-\phi\mathcal{T}(x,|f_n-f|)-\phi\mathcal{T}(x,|f|)\right|\diff x=0.$$	
\end{lemma}	
%========================PROOOF OF THE CCP=========================%
We are now in a position to prove Theorem~\ref{Theo.ccp}.
\begin{proof}[\textbf{Proof of Theorem~\ref{Theo.ccp}}]
Let $v_n=u_n-u$. Then, up to a subsequence, we have
\begin{eqnarray}\label{T.conv.of.v_n}
	\begin{cases}
		v_n(x) &\to  \quad 0 \quad \text{a.a.}\quad  x\in\Omega,\\
		v_n &\rightharpoonup \quad 0 \quad \text{in}\quad  W_0^{1,\mathcal{H}}(\Omega).
	\end{cases}
\end{eqnarray}
Applying Lemma~\ref{L.brezis-lieb} with $\Phi=\mathcal{B}$, we deduce
$$
\lim_{n\to\infty}\int_\Omega\left|\phi\mathcal{B}(x,u_n)-\phi\mathcal{B}(x,v_n) - \phi\mathcal{B}(x,u)\right|\diff x=0,\quad \forall \phi \in C(\ol{\Omega}).$$
Thus,
\begin{equation}\label{T.w*-vn}
\bar{\nu}_n:=	\mathcal{B}(\cdot,v_n)\quad\overset{\ast }{\rightharpoonup}\quad\bar{\nu}:=\nu-\mathcal{B}(\cdot,u)\quad \text{in} \ \ \mathcal{M}(\overline{\Omega}).
\end{equation}		
	It is clear that $\{\mathcal{H}(\cdot,|\nabla v_n|)\}_{n\in\N}$ is bounded in $L^1(\Omega)$ due to the boundedness of $\{u_n\}_{n\in\N}$ in $W_0^{1,\mathcal{H}}(\Omega)$. So up to a subsequence, we have
\begin{equation}\label{bar.mu}
	\bar{\mu}_n:=\mathcal{H}(\cdot,|\nabla v_n|)\quad \overset{\ast }{\rightharpoonup }\quad \bar{\mu}\quad \text{in}\quad \mathcal{M}(\overline{\Omega})
\end{equation}
for some finite nonnegative Radon measure $\bar{\mu}$ on $\overline{\Omega}$. 

In view of Lemma~\ref{L.reserveHolder}, \eqref{T.ccp.form.nu} will be proved if we can show that
\begin{equation}\label{PT1.1.RH}
	\|\phi\|_{{L}^{r_1(\cdot)}_{\bar{\nu}}(\overline{\Omega}) }\leq C \|\phi\|_{{L}^{q(\cdot)}_{\bar{\mu}}(\overline{\Omega}) },\quad \forall \phi\in C^\infty(\overline{\Omega})
\end{equation}
for some positive constant $C$. To this end, let $\phi\in C^\infty(\overline{\Omega})$. It is clear that $\phi v\in W_0^{1,\mathcal{H}}(\Omega)$ for any $v\in W_0^{1,\mathcal{H}}(\Omega)$; hence, by invoking \eqref{S} we obtain
\begin{align}
S\|\phi v_n\|_{\mathcal{B}} \leq \norm{\nabla(\phi v_n)}_{\mathcal{H}}\leq \norm{\phi\nabla v_n}_{\mathcal{H}}+\norm{v_n\nabla\phi}_{\mathcal{H}}\leq \|\phi \nabla v_n\|_{\mathcal{H}} +\|\phi\|_{C^1(\overline{\Omega})}\|v_n\|_{{\mathcal{H}}}.\label{T.est.norm1}
\end{align} 	
Set  $\bar{\lambda}_n:=\|\phi\|_{L^{r_1(\cdot)}_{\bar{\nu}_n}(\overline{\Omega})}$ for $n\in\N$ and $\bar{\lambda}:=\|\phi\|_{L^{r_1(\cdot)}_{\bar{\nu}}(\overline{\Omega})}$. Then, by Lemma~\ref{L.convergence} and \eqref{T.w*-vn}, it holds 
\begin{equation}\label{lambda_n_bar}
	\lim_{n\to\infty}\bar{\lambda}_n=\bar{\lambda}.
\end{equation}
 Clearly, \eqref{PT1.1.RH} holds for the case $\bar{\lambda}=0$. Let us consider the case $\bar{\lambda}>0$, and then we can suppose that $\bar{\lambda}_n>0$ for all $n \in \mathbb{N}$.
 By the boundedness of $\{u_n\}_{n\in\N}$ in $W_0^{1,\mathcal{H}}(\Omega)$, from Proposition~\ref{prop_S-C-E} it holds
\begin{equation}\label{vn_M}
M:=1+\max \bigg\{\sup_{n\in \mathbb{N}}\int_{\Omega} \mathcal{B}(x,v_n) \diff x,\,	\sup_{n\in \mathbb{N}}\int_{\Omega} \mathcal{H}(x,|\nabla v_n|) \diff x \bigg \}\in [1,\infty).
\end{equation} 
Invoking Proposition~\ref{prop.nor-mod.M} and \eqref{young} we find $C_M>1$ such that
\begin{align*}
	\notag
1=\int_{\overline{\Omega}} \left|\frac{\phi}{\bar{\lambda}_n}\right|^{r_1(x)} \diff \bar{\nu}_n &=\int_{\Omega} \left|\frac{\phi}{\bar{\lambda}_n}\right|^{r_1(x)}\bigg[c_1(x)|v_n|^{r_1(x)}+c_2(x)a(x)^{\frac{r_2(x)}{q(x)}}|v_n|^{r_2(x)}\bigg]\diff x \\
&\leq \int_{\Omega} c_1(x)\left|\frac{\phi v_n}{\bar{\lambda}_n}\right|^{r_1(x)}\diff x+\int_{\Omega}c_2(x)a(x)^{\frac{r_2(x)}{q(x)}}|v_n|^{r_2(x)}\left(\frac{1}{2M}+C_M\left|\frac{\phi}{\bar{\lambda}_n}\right|^{r_2(x)}  \right)\diff x \notag \\
 & \leq \frac{1}{2M} \int_{\Omega} \mathcal{B}\left(x,v_n\right)\diff x +C_M\int_{\Omega} \mathcal{B}\left(x,\frac{\phi v_n}{\bar{\lambda}_n}\right) \diff x.
\end{align*}
From this, \eqref{vn_M} and the definition of $\mathcal{B}$ we easily obtain
\begin{align}
	\notag
	1\leq\int_{\Omega} \mathcal{B}\left(x,\frac{(2C_M)^{\frac{1}{r_1^-}}\phi v_n}{\bar{\lambda}_n}\right) \diff x.
\end{align}
Thus, in view of Proposition~\ref{prop.nor-mod.D} we obtain
\begin{align*}
(2C_M)^{-\frac{1}{r_1^-}}\bar{\lambda}_n\leq \|\phi v_n\|_{\mathcal{B}},\quad \forall n \in \mathbb{N}.
\end{align*}
Combining this with \eqref{lambda_n_bar} gives
\begin{equation}\label{T.norm.left}
0<(2C_M)^{-\frac{1}{r_1^-}}\|\phi\|_{L^{r_1(\cdot)}_{\bar{\nu}}(\overline{\Omega})}\leq\liminf_{n \to \infty}\|\phi v_n\|_{\mathcal{B}}.
\end{equation}
Next, we set $\delta_n:=\|\phi \nabla v_n\|_{\mathcal{H}}$ for $n\in\N$. Note that $v_n\to 0$ in $L^{\mathcal{H}}(\Omega)$ due to \eqref{T.conv.of.v_n} and Proposition~\ref{prop_S-C-E}. By this, \eqref{T.est.norm1} and \eqref{T.norm.left} we may assume that  $\delta_n>0$ for all $n \in \mathbb{N}$.  Employing Proposition~\ref{prop.nor-mod.D} and \eqref{young} again we find $\bar{C}_M>1$ such that 
\begin{align*}
	\notag
	1&=\int_{\Omega} \round{ \left|\frac{\phi \nabla v_n}{\delta_n}\right|^{p(x)} + a(x)\left|\frac{\phi \nabla v_n}{\delta_n}\right|^{q(x)}} \diff x\\
	& \leq  \int_\Omega \left(\frac{1}{2M}+\bar{C}_M\left|\frac{\phi}{\delta_n}\right|^{q(x)}\right)|\nabla v_n|^{p(x)} \diff x + \int_{\Omega} \left|\frac{\phi}{\delta_n}\right|^{q(x)} a(x)|\nabla v_n|^{q(x)} \diff x.
	\end{align*}
From this and \eqref{vn_M} it follows that
\begin{equation*}
	1 \leq \int_{\overline{\Omega}}  \left|\frac{(2\bar{C}_M)^{\frac{1}{q^-}}\phi}{\delta_n}\right|^{q(x)} \diff \bar{\mu}_n.
\end{equation*}
Hence, by Proposition~\ref{prop.nor-mod.M} we obtain
\begin{equation*}
	(2\bar{C}_M)^{-\frac{1}{q^-}}\delta_n \leq  \|\phi\|_{L^{q(\cdot)}_{\bar{\mu}_n}(\overline{\Omega})}.
\end{equation*}
By virtue of Lemma~\ref{L.convergence} we deduce from the preceding inequality and \eqref{bar.mu} that
\begin{equation} \label{T.norm.right}
	\limsup_{n \to \infty}\|\phi \nabla v_n\|_{\mathcal{H}} \leq (2\bar{C}_M)^{\frac{1}{q^-}} \|\phi\|_{L^{q(\cdot)}_{\bar{\mu}}(\overline{\Omega})}.
\end{equation}
Utilizing \eqref{T.norm.left}, \eqref{T.norm.right} and the fact that $v_n\to 0$ in $L^{\mathcal{H}}(\Omega)$, the desired inequality \eqref{PT1.1.RH} can easily obtain from \eqref{T.est.norm1}; hence, \eqref{T.ccp.form.nu} has been proved.

Next, we claim that $\{x_i\}_{i\in I}\subset \mathcal{C}$. Assume by contradiction that there is some $x_i\in \overline{\Omega}\setminus\mathcal{C}$. Let $\delta>0$ be such that $\overline{B_{2\delta}(x_i)}\subset \mathbb{R}^N\setminus\mathcal{C}$. Then, by setting $B:=B_\delta(x_i)\cap \overline{\Omega}$, it holds  $\overline{B}\subset \overline{\Omega}\setminus\mathcal{C}$; hence, $r_1(x)<p^\ast(x)$ and  $r_2(x)=q^\ast(x)-(p^\ast(x)-r_1(x))<q^\ast(x)$  for all $x\in\overline{B}$. Consequently, 
$$\int_{B}\mathcal{B}(x,u_n)\diff x\to \int_{B}\mathcal{B}(x,u)\diff x$$
in view of Proposition~\ref{prop_S-C-E}. From this and the fact that $\nu (B)\leq \liminf_{n\to \infty}\int_{B}\mathcal{B}(x,u_n)\diff x$ (see \cite[Proposition 1.203]{FL07}), we obtain $\nu (B)\leq \int_{B}\mathcal{B}(x,u) \diff x.$
On the other hand, by \eqref{T.ccp.form.nu} we have
$$\nu (B)\geq \int_{B}\mathcal{B}(x,u)\diff x+\nu_i>\int_{B}\mathcal{B}(x,u)\diff x,$$
 a contradiction. So, $\{x_i\}_{i\in I}\subset \mathcal{C}$.

In order to show \eqref{T.ccp.nu_mu}, let $i\in I$ and let $\eta$ be in $C_c^\infty(\mathbb{R}^N)$ such that $0\leq \eta(\cdot)\leq 1,$ $\eta(\cdot)\equiv 1$ on $B_{1/2}(0)$ and $\eta(\cdot)\equiv 0$ outside $B_1(0)$. For $\epsilon>0$ and $h\in C(\overline{\Omega})$, we denote
\begin{gather*}
\Omega_{i,\epsilon}:=B_\epsilon(x_i)\cap \overline{\Omega},\quad h^+_{\epsilon}:=\sup_{x\in \Omega_{i,\epsilon}}h(x),\quad
h^-_{\epsilon}:=\inf_{x\in \Omega_{i,\epsilon}}h(x)
\end{gather*}
and define
$$\phi_{i,\epsilon}(x):=\eta\left(\frac{x-x_i}{\epsilon}\right) \quad\text{for} \ x\in\R^N.$$
Using  \eqref{S} again with $\phi=\phi_{i,\epsilon}u_n$, we have
\begin{align} \label{RH2}
	\notag
	S\|\phi_{i,\epsilon} u_n\|_{{\mathcal{B}}}& \leq \|\nabla (\phi_{i,\epsilon} u_n)\|_{\mathcal{H}}  \\ \notag
	& \leq \|\phi_{i,\epsilon}\nabla u_n\|_{\mathcal{H}}+\|u_n \nabla \phi_{i,\epsilon}\|_{\mathcal{H}}\\  \notag
	& \leq \|\phi_{i,\epsilon}\nabla u_n\|_{\mathcal{H}}+\|(u_n-u) \nabla \phi_{i,\epsilon}\|_{\mathcal{H}}+\|u \nabla \phi_{i,\epsilon}\|_{\mathcal{H}}\\ 
	& \leq \|\phi_{i,\epsilon}\nabla u_n\|_{\mathcal{H}}+\|\phi_{i,\epsilon}\|_{C^1(\overline{\Omega})}\norm{u_n-u}_{\mathcal{H}}+\|u \nabla \phi_{i,\epsilon}\|_{\mathcal{H}}.
\end{align}
Invoking Proposition~\ref{prop.nor-mod.D} we obtain
\begin{align*}
\|\phi_{i,\epsilon} u_n\|_{{\mathcal{B}}}&\geq \min\left\{\left(\int_{\Omega_{i,\epsilon}} \mathcal{B}\left(x,\phi_{i,\epsilon} u_n\right)\diff x\right)^{\frac{1}{(r_1)_{\epsilon}^-}},\left(\int_{\Omega_{i,\epsilon}} \mathcal{B}\left(x,\phi_{i,\epsilon} u_n\right)\diff x\right)^{\frac{1}{(r_2)_{\epsilon}^+}}\right\}\\
	&\geq \min\left\{\left(\int_{\Omega_{i,\epsilon/2}} \mathcal{B}\left(x,u_n\right)\diff x\right)^{\frac{1}{(r_1)_{\epsilon}^-}},\left(\int_{\Omega_{i,\epsilon/2}} \mathcal{B}\left(x,u_n\right)\diff x\right)^{\frac{1}{(r_2)_{\epsilon}^+}}\right\}.
\end{align*}
From this and \eqref{ccp.nu}, we arrive at
\begin{align}\label{nu_i}
\liminf_{ n \to \infty} \|\phi_{i,\epsilon} u_n\|_{{\mathcal{B}}} \geq \min \left\{ 
\nu(\Omega_{i,\epsilon/2})^{\frac{1}{(r_1)_{\epsilon}^-}}, \nu(\Omega_{i,\epsilon/2})^{\frac{1}{(r_2)_{\epsilon}^+}}
\right\}.
\end{align}
Similarly, we get that
\begin{align*}
	\|\phi_{i,\epsilon}\nabla u_n \|_{\mathcal{H}}&\leq \max\left\{\left(\int_{\Omega} \mathcal{H}(x,\phi_{i,\epsilon}|\nabla u_n|)\diff x\right)^{\frac{1}{p_{\epsilon}^-}},\left(\int_{\Omega}\mathcal{H}(x,\phi_{i,\epsilon}|\nabla u_n|)\diff x\right)^{\frac{1}{q_{\epsilon}^+}}\right\}\\
	&\leq \max\left\{\left(\int_{\Omega} \phi_{i,\epsilon}\mathcal{H}(x,|\nabla u_n|)\diff x\right)^{\frac{1}{p_{\epsilon}^-}},\left(\int_{\Omega}\phi_{i,\epsilon}\mathcal{H}(x,|\nabla u_n|)\diff x\right)^{\frac{1}{q_{\epsilon}^+}}\right\}.
\end{align*}
From this and \eqref{ccp.mu}, we arrive at
\begin{align*}
	\limsup_{ n \to \infty} \|\phi_{i,\epsilon}\nabla u_n \|_{\mathcal{H}} \leq \max\left\{\left(\int_{\overline{\Omega}} \phi_{i,\epsilon}\diff \mu\right)^{\frac{1}{p_{\epsilon}^-}},\left(\int_{\overline{\Omega}} \phi_{i,\epsilon}\diff \mu\right)^{\frac{1}{q_{\epsilon}^+}}\right\}.
\end{align*}
Thus, we obtain
\begin{equation}\label{mu_i}
\limsup_{n \to \infty}\| \phi_{i,\epsilon}\nabla u_n \|_{\mathcal{H}} \leq \max \left \{\mu (\Omega_{i,\epsilon})^{\frac{1}{p_{\epsilon}^-}}, \mu (\Omega_{i,\epsilon})^{\frac{1}{q_{\epsilon}^+}}  \right \}.
\end{equation}
Moreover, in view of the compact embedding \eqref{Prop-S-C} we deduce from \eqref{ccp.weak} that
\begin{equation}\label{P.CCP.lim_un}
	u_n\to u \ \ \text{in} \ \ L^{\mathcal{H}}(\Omega).
\end{equation}
Passing to the limit as $n\to \infty$ in \eqref{RH2} and taking into account \eqref{nu_i}-\eqref{P.CCP.lim_un}, we arrive at 
\begin{equation}\label{T.estforsingular}
S \min \left\{ 
\nu(\Omega_{i,\epsilon/2})^{\frac{1}{(r_1)_{\epsilon}^-}}, \nu(\Omega_{i,\epsilon/2})^{\frac{1}{(r_2)_{\epsilon}^+}}
\right\}\leq \max \left \{\mu (\Omega_{i,\epsilon})^{\frac{1}{p_{\epsilon}^-}}, \mu (\Omega_{i,\epsilon})^{\frac{1}{q_{\epsilon}^+}}  \right \} + 
\|u \nabla \phi_{i,\epsilon}\|_{\mathcal{H}}.
\end{equation}
On the other hand, by Proposition~\ref{prop_S-C-E} it holds  $u\in L^{\mathcal{G}^\ast}(\Omega)$, where $\mathcal{G}^\ast(x,t)=|t|^{p^\ast(x)}+a(x)^{\frac{q^*(x)}{q(x)}}|t|^{q^\ast(x)}$ for $(x,t)\in \overline{\Omega}\times \R$. Using this fact and Proposition~\ref{prop.Holder} we obtain 
\begin{align}\label{est.critical.1}
\notag\int_{\Omega} \mathcal{H}(x,|u\nabla \phi_{i,\epsilon}|) \diff x&=\int_{\Omega_{i,\epsilon}} \left[|u\nabla \phi_{i,\epsilon}|^{p(x)}+a(x)|u\nabla \phi_{i,\epsilon}|^{q(x)}\right] \diff x\\
\notag&\leq 2\big\||u|^{p(\cdot)}\big\|_{L^{\frac{p^\ast(\cdot)}{p(\cdot)}}(\Omega_{i,\epsilon})}
\big\||\nabla \phi_{i,\epsilon}|^{p(\cdot)} \big\|_{L^{\frac{N}{p(\cdot)}}
	(B_{\epsilon}(x_i))}\\
&\qquad\qquad\qquad+2\big\|\,a|u|^{q(\cdot)}\,\big\|_{L^{\frac{q^\ast(\cdot)}{q(\cdot)}}(\Omega_{i,\epsilon})}\big\||\nabla \phi_{i,\epsilon}|^{q(\cdot)} \big\|_{L^{\frac{N}{q(\cdot)}}
	(B_{\epsilon}(x_i))}.
\end{align}
By Proposition~\ref{prop.nor-mod.D}, it follows that
$$
\big\||\nabla \phi_{i,\epsilon}|^{p(\cdot)} \big\|_{L^{\frac{N}{p(\cdot)}}
	(B_{\epsilon}(x_i))}
\le \round{1+ \int_{B_\epsilon(x_i)} |\nabla \phi_{i,\epsilon}|^N \diff x}^{p^+/N}  =\round{ 1+\int_{B_1(0)} |\nabla\eta(y)|^N\, \diff y}^{p^+/N},
$$
and in the same manner, we also have
$$ \big\||\nabla \phi_{i,\epsilon}|^{q(\cdot)}\big\|_{L^{\frac{N}{q(\cdot)}}(B_{\epsilon}(x_i))}\leq \round{ 1+\int_{B_1(0)} |\nabla\eta(y)|^N\, \diff y}^{q^+/N}.$$
Utilizing the last two estimates we infer from  \eqref{est.critical.1} that
\begin{equation*}
	\int_{\Omega} \mathcal{H}(x,|u\nabla \phi_{i,\epsilon}|) \diff x\to 0\ \ \text{as}\ \ \epsilon \to 0^+.
\end{equation*}
Equivalently, in view of Proposition~\ref{prop.nor-mod.D}, it holds
\begin{equation}\label{est.critical.2}
\|u \nabla \phi_{i,\epsilon}\|_{\mathcal{H}}\to 0\ \ \text{as}\ \ \epsilon \to 0^+.
\end{equation}
Passing to the limit as $\epsilon \to 0^+$ in \eqref{T.estforsingular} and taking into account  \eqref{est.critical.2}, the continuity of $r_1,r_2,q,p$ on $\overline{\Omega},$ and the fact that $x_i\in  \mathcal{C}$ we obtain
$$S \min\left\{\nu_i^{1/p^\ast(x_i)},\nu_i^{1/q^\ast(x_i)}\right\} \le \max\left\{\mu_i^{1/p(x_i)},\mu_i^{1/q(x_i)}\right\},
$$
where $\mu_i := \mu(\{x_i\})$. In particular, $\{x_i\}_{i\in I}$ are atoms of $\mu$. 

Finally, in order to get \eqref{T.ccp.form.mu}, noticing that for any $\phi\in C(\overline{\Omega})$  with $\phi\geq 0$, the functional $u\mapsto \int_{\Omega}\phi(x)\left[|\nabla u|^{p(x)}+a(x)|\nabla u|^{q(x)}\right]\diff x$ is convex and differentiable on $W_0^{1,\mathcal{H}}(\Omega)$. Hence, it is weakly lower semicontinuous and therefore,
$$\int_{\overline{\Omega }}\phi(x)\left[|\nabla u|^{p(x)}+a(x)|\nabla u|^{q(x)}\right]\diff x\leq \liminf_{n\to  \infty } \int_{\overline{\Omega } }\phi(x)\left[|\nabla u_n|^{p(x)}+a(x)|\nabla u_n|^{q(x)}\right]\diff x=\int_{\overline{\Omega }}\phi \diff \mu.$$
Thus, $\mu\geq |\nabla u|^{p(x)}+a(x)|\nabla u|^{q(x)}=\mathcal{H}(x,|\nabla u|)$. By extracting $\mu$ to its atoms, we deduce \eqref{T.ccp.form.mu}. The proof is complete.
\end{proof}
%=====================SECTION 4. APPLICATIONS========================== %
\section{Proofs of Multiplicity Results}\label{Existence}

In this section, we will prove the multiplicity results stated in of Section ~\ref{MainResults}. In the sequel, we denote $X:=\left(W_0^{1,\mathcal{H}}(\Omega),\|\cdot\|\right)$ with $\|\cdot\|$ given by \eqref{norm} and denote by $X^*$ the dual space of $X$. %We also denote the dual space of $X$ and its norm by  $X^\ast$ and $\|\cdot\|_*$, respectively. 
We also denote by $B_\tau$ and $|Q|$ the open ball in $X$ centered at $0$ with radius $\tau$ and the Lebesgue measure of $Q\subset\R^N$, respectively. %To seek solution of \eqref{e1.1'}, we consider the associated  energy functional as follows
%\begin{align}\label{J_lambda}
%	\notag J_\lambda(u)& :=\widehat{M}\left(\int_\Omega \mathcal{A}(x,|\nabla u|) \diff x\right)-\lambda\int_\Omega F(x,u)\diff x-\theta\int_\Omega \widehat{B}(x,u)\diff x,\quad u\in X,
%\end{align}		
%where $\widehat{M}(t):=\int_0^t M(s)\diff s$, $\widehat{B}(x,t):= \int_0^t B(x,s)\diff s$. We can see clearly that  $J_\lambda: X\to\R$ is of class $C^1(X,\mathbb{R})$. 
\subsection{The generalized concave-convex case}${}$

In this subsection we will prove Theorem~\ref{Theo.cc} via the genus theory along with a truncation technique used in \cite{BGP96,FN16} as follows. Let all hypotheses in Theorem~\ref{Theo.cc} hold and let $\theta=1$.
From \eqref{q-r1} and the continuity of $M$ on $(0,\tau_0)$, we can fix  $t_0\in (0,\min\{\tau_0,1\})$ such that
\begin{equation}\label{t0-sub}
	M(t_0)q^+<m_0r_1^-.
\end{equation} 
Define a truncation of $M$ as
\begin{equation}\label{Def.M0}
	M_0(t):=\begin{cases}
		M(t),\ \ &0\leq t\leq t_0,\\
		M(t_0),\ \ &t>t_0
	\end{cases}
\end{equation}
and define 
$$\widehat{M}_0(t):=\int_0^t M_0(s)\diff s\ \  \text{for}\ \  t\geq 0.$$ Clearly, $M_0$ is a  continuous function on $[0,+\infty)$ and satisfies
\begin{equation}\label{Est.M0}
	m_0\leq M_0(t)\leq M(t_0),\quad \forall t\in [0,+\infty)
\end{equation}		
and 
\begin{equation}\label{Est.hat.M0}
	m_0t\leq \widehat{M}_0(t)\leq M(t_0)t,\quad \forall t\in[0,+\infty).
\end{equation}
For each $\lambda>0$, define the modified energy functional $\Phi_\lambda:X \to \mathbb{R}$ as
\begin{align}\label{Phi-Lam}
	 \Phi_\lambda(u):=&\widehat{M}_0\left(\int_\Omega \mathcal{A}(x,\nabla u) \diff x\right)-\lambda\int_\Omega F(x,u) 
	\diff x-\int_\Omega \widehat{B}(x,u) \diff x,\quad u\in X,
\end{align}
where 
\begin{equation}\label{B.hat}
	\widehat{B}(x,t):=\int_0^t B(x,s)\diff s.
\end{equation}
By a standard argument, we can show that $\Phi_\lambda$ is of class $C^1$ in view of Proposition~\ref{prop_S-C-E}, and its Fr\'echet derivative $\Phi^\prime_\lambda: X\to X^\ast$ is given by 
\begin{multline}\label{Diff1}
\left\langle \Phi^\prime_\lambda (u),v\right\rangle=  M_0\left(\int_\Omega\mathcal{A}(x,\nabla u)\,\diff x\right)\int_\Omega A(x,\nabla u)\cdot\nabla v\,\diff x\\
-\lambda\int_\Omega f(x,u)v\,\diff x-\int_\Omega B(x,u)v\,\diff x,\quad \forall\, u,v\in X.
\end{multline}
 %Moreover, $\Phi_\lambda$ is even on $X$ and $\Phi_\lambda(0)=0.$
  Obviously, any critical point $u$ of $\Phi_\lambda$ with $\int_\Omega \mathcal{A}(x,\nabla u)\,\diff x\leq t_0$ is a solution to problem~\eqref{e1.1'}.

Before going to some auxiliary results, we introduce several notations for simplicity. From \eqref{t0-sub} and $(\mathcal P)$ we have
\begin{equation}\label{K0}
	K_0:=\frac{1}{4}\left(\frac{m_0}{q^+}-\frac{M(t_0)}{r_1^-}\right)>0
\end{equation}
and 
\begin{equation}\label{lambda1}
	\lambda_1:=\frac{\kappa_1r_1^- K_0}{C_4+C_5}>0,
\end{equation}
with $C_4$ and $C_5$ given in $(\mathcal{F}_2)$. 

The next lemma is essential to verify the (local) $\textup{(PS)}$ condition for $\Phi_\lambda$.
\begin{lemma} \label{Le.PS1}
Let $\lambda \in (0,\lambda_1)$. Then,  $\Phi_\lambda$ satisfies the $\textup{(PS)}_c$ condition with $c\in \mathbb{R}$ satisfying 
\begin{align}\label{Le.PS1.c}
	c&<K_0\min\{S^{n_1},S^{n_2}\}\min\{m_0^{\tau_1},m_0^{\tau_2}\}-K\max\left\{\lambda^{\frac{\ell^+}{\ell^+-1}},\lambda^{\frac{\ell^-}{\ell^--1}}\right\},
\end{align}
where $S$ is defined in \eqref{S},  $n_1:=\left(\frac{pq^*}{q^*-p}\right)^-$, $n_2:=\left(\frac{qp^*}{p^*-q}\right)^+$,  $\tau_1:=\left(\frac{p}{q^*-p}\right)^-$, $\tau_2:=\left(\frac{q}{p^*-q}\right)^+$, $\ell(\cdot):=\frac{p(\cdot)}{\sigma(\cdot)}$, and $K$ is a positive constant depending only on the data given by \eqref{K} below.
\end{lemma}
%=================PROOF OF LEMMA 4.8. (PS) CONDITION=================%
\begin{proof}
Let $\lambda \in (0,\lambda_1)$ and let $\{u_n\}_{n\in\N}$ be a $\textup{(PS)}_c$ sequence in $X$ for $\Phi_\lambda$ with $c$ satisfying \eqref{Le.PS1.c}, namely,
\begin{equation}\label{Le.PS1.PS-seq}
	\Phi_\lambda(u_n)\to c\ \ \text{and}\ \ \Phi_\lambda'(u_n)\to 0.
\end{equation}	 
We verify the boundedness of $\{u_n\}_{n\in\N}$ in $X$. To this end, it suffices to argue with $n$ large. From \eqref{Est.M0}, \eqref{Est.hat.M0} and \eqref{Le.PS1.PS-seq}, we have 
\begin{align*}
\notag
c+1+\|u_n\|& \geq \Phi_\lambda(u_n)-\frac{1}{r_1^-}\scal{\Phi^\prime_\lambda(u_n),u_n}\\ \notag
&\geq m_0\int_\Omega \mathcal{A}(x,\nabla u) \diff x-\frac{M(t_0)}{r_1^-}\int_\Omega \mathcal{H}(x,|\nabla u_n|)\diff x-\frac{\lambda}{r_1^-}\int_\Omega \left[r_1^-F(x,u_n)-f(x,u_n)u_n\right]\diff x \notag\\
&\geq \round{\frac{m_0}{q^+}-\frac{M(t_0)}{r_1^-}}\int_\Omega \mathcal{H}(x,|\nabla u_n|)\diff x-\frac{\lambda}{r_1^-}\int_\Omega \left[r_1^-F(x,u_n)-f(x,u_n)u_n\right]\diff x. 
\end{align*}
Then, utilizing $(\mathcal{F}_2)$, \eqref{mu1}, \eqref{K0} and \eqref{lambda1}  we deduce from the last estimate that
\begin{align}\label{PL.PS1.E1}
	\notag
	c+1+\|u_n\|&\geq 4K_0\int_\Omega \mathcal{H}(x,|\nabla u_n|)\diff x-\frac{\lambda}{r_1^-}\int_\Omega \left[C_4|u_n|^{\sigma(x)}+C_5|u_n|^{p(x)}\right] \notag\\
		&\geq 4K_0\int_\Omega \mathcal{H}(x,|\nabla u_n|)\diff x-\frac{\lambda}{r_1^-}\int_\Omega \left[(C_4+C_5)|u_n|^{p(x)}+C_4\right] \notag\\ 
	&\geq \round{4K_0-\frac{\lambda (C_4+C_5)}{\kappa_1r_1^-}}\int_\Omega \mathcal{H}(x,|\nabla u_n|)\diff x-\frac{\lambda C_4|\Omega|}{r_1^-}\notag\\
	&\geq 3K_0\int_\Omega \mathcal{H}(x,|\nabla u_n|)\diff x-\frac{\lambda C_4|\Omega|}{r_1^-}.
\end{align}
Invoking Proposition~\ref{prop.nor-mod.D}, we deduce from \eqref{PL.PS1.E1} that
$$c+1+\|u_n\| \geq 3K_0\left(\|u_n\|^{p^-}-1\right)-\frac{\lambda C_4|\Omega|}{r_1^-},$$
which leads to the boundedness of $\{u_n\}_{n\in\N}$ in $X$ since $p^->1$.  Then by Theorem~\ref{Theo.ccp}, up to a subsequence, we have
\begin{eqnarray}
	u_n(x) &\to& u(x)  \quad \text{a.a.} \ \ x\in\Omega,\label{PL.PS1.a.e}\\
	u_n &\rightharpoonup& u  \quad \text{in} \  X,\label{PL.PS1.w-conv}\\
	\mathcal{H}(\cdot,|\nabla u_n|)\ &\overset{\ast }{\rightharpoonup }&\mu \geq	\mathcal{H}(\cdot,|\nabla u|) + \sum_{i\in I} \mu_i \delta_{x_i} \ \text{in}\  \mathcal{M}(\overline{\Omega}),\label{PL.PS1.mu}\\
	\mathcal{B}(\cdot,u_n)&\overset{\ast }{\rightharpoonup }&\nu=\mathcal{B}(\cdot,u) + \sum_{i\in I}\nu_i\delta_{x_i} \ \text{in}\ \mathcal{M}(\overline{\Omega}),\label{PL.PS1.nu}\\
	S \min\left\{\nu_i^{\frac{1}{p^\ast(x_i)}},\nu_i^{\frac{1}{q^\ast(x_i)}}\right\} &\leq& \max\left\{\mu_i^{\frac{1}{p(x_i)}},\mu_i^{\frac{1}{q(x_i)}}\right\} , \ \forall i\in I.\label{PL.PS1.mu-nu}
\end{eqnarray}	
We claim that $I = \emptyset.$ Assume on the contrary that there exists $i\in I$. Let $\epsilon>0$ and define $\phi_{i,\epsilon}$ as in the proof of Theorem~\ref{Theo.ccp}. Since $\phi_{i,\epsilon}u_n \in X$, we get
\begin{multline}\label{Deri.J}	
M_0\round{\int_\Omega\mathcal{A}(x,\nabla u_n)\diff x}\int_\Omega \phi_{i,\epsilon}\mathcal{H}(x,|\nabla u_n|)\diff x
=\langle \Phi_\lambda'(u_n) ,\phi_{i,\epsilon}u_n \rangle+\lambda\int_{\Omega}\phi_{i,\epsilon}f(x,u_n)u_n\diff x\\
+\int_\Omega\phi_{i,\epsilon}\mathcal{B}(x,u_n)\diff x -M_0\round{\int_\Omega \mathcal{A}(x,\nabla u_n)\diff x}\int_\Omega A(x,\nabla u_n)\cdot \nabla \phi_{i,\epsilon}u_n \diff x.	
\end{multline}
Let $\delta>0$ be arbitrary. Applying \eqref{young}, we have
\begin{align*}\label{nablaPhi.Young}
	\int_\Omega \big|A(x,\nabla u_n)\cdot \nabla \phi_{i,\epsilon}u_n\big| \diff x \leq& \int_\Omega |\nabla u_n|^{p(x)-1}|\nabla \phi_{i,\epsilon}||u_n|\diff x  +\int_\Omega a(x)|\nabla u_n|^{q(x)-1}|\nabla \phi_{i,\epsilon}||u_n|\diff x \notag\\
	\leq &  \delta \int_\Omega \mathcal{H}(x,|\nabla u_n|)\diff x+C_\delta \int_\Omega \mathcal{H}(x,|\nabla \phi_{i,\epsilon}u_n|)\diff x\notag\\
		\leq &  C_\star\delta +C_\delta \int_\Omega \mathcal{H}(x,|\nabla \phi_{i,\epsilon}u_n|)\diff x, 
\end{align*}	
where $C_\delta$ is a positive constant depending only on the data and $\delta$ while
\begin{equation*}
	C_\star:=\sup_{n\in \mathbb{N}}\int_\Omega \mathcal{H}(x,|\nabla u_n|)\diff x<\infty.
\end{equation*}
Combining this with \eqref{Deri.J} and taking into account \eqref{Est.M0} we obtain
\begin{multline}\label{Deri.J'}	
	m_0\int_\Omega \phi_{i,\epsilon}\mathcal{H}(x,|\nabla u_n|)\diff x
	\leq \langle \Phi_\lambda'(u_n) ,\phi_{i,\epsilon}u_n \rangle+\lambda\int_{\Omega}\phi_{i,\epsilon}f(x,u_n)u_n\diff x\\
	+\int_\Omega\phi_{i,\epsilon}\mathcal{B}(x,u_n)\diff x +C_\star M(t_0)\delta+M(t_0)C_\delta\int_\Omega \mathcal{H}(x,|\nabla \phi_{i,\epsilon}u_n|)\diff x.	
\end{multline}
Note that $\{\phi_{i,\epsilon}u_n\}_{n\in\N}$ is bounded in $X$. From this and \eqref{Le.PS1.PS-seq} we obtain
\begin{equation}\label{Deri.J'1}	
	\lim_{n\to\infty}\langle \Phi_\lambda'(u_n) ,\phi_{i,\epsilon}u_n \rangle=0.
\end{equation}
By Proposition~\ref{prop_S-C-E}, it follows from \eqref{PL.PS1.w-conv} that $u_n\to u$ in $L^{\mathcal{H}}(\Omega)$. From this we easily obtain 
\begin{equation}\label{Deri.J'2}	
	\lim_{n\to\infty}\int_\Omega \mathcal{H}(x,|\nabla \phi_{i,\epsilon}u_n|)\diff x=\int_\Omega \mathcal{H}(x,|\nabla \phi_{i,\epsilon}u|)\diff x.
\end{equation}
On the other hand, the assumption $(\mathcal{F}_0)$ and Proposition~\ref{prop_embs} imply that
\begin{equation}\label{Deri.J'3}	
\lim_{n\to\infty}\int_\Omega \phi_{i,\epsilon} f(x,u_n)u_n \diff x= \int_\Omega \phi_{i,\epsilon} f(x,u)u \diff x.
\end{equation}
Passing to the limit as $n \to \infty$ in \eqref{Deri.J'}, taking into account \eqref{PL.PS1.mu}, \eqref{PL.PS1.nu} and   \eqref{Deri.J'1}-\eqref{Deri.J'3},  we obtain
\begin{equation}\label{PL.PS1.I}
	m_0\int_{\overline{\Omega}}  \phi_{i,\epsilon} \diff \mu \leq \lambda \int_{\Omega} \phi_{i,\epsilon} f(x,u)u \diff x + \int_{\overline{\Omega}} \phi_{i,\epsilon} d\nu +C_\star M(t_0)\delta+M(t_0)C_\delta\int_\Omega \mathcal{H}(x,|\nabla \phi_{i,\epsilon}u|)\diff x.
\end{equation}	
By the integrability of $f(\cdot,u)u$ and arguing as those leading to \eqref{est.critical.2} we obtain
\begin{equation}\label{PL.PS1.I1}
	\lim_{\epsilon\to 0^+}\int_{\Omega} \phi_{i,\epsilon} f(x,u)u \diff x =\lim_{\epsilon\to 0^+}\int_\Omega \mathcal{H}(x,|\nabla \phi_{i,\epsilon}u|)\diff x=0.
\end{equation}
Then, passing to the limit as $\epsilon \to 0^+$ in \eqref{PL.PS1.I} and utilizing \eqref{PL.PS1.I1} we arrive at
\begin{equation*}
	m_0 \mu_i \leq \nu_i + C_\star M(t_0)\delta.
\end{equation*}
Since $\delta$ was chosen arbitrarily, the last estimate yields
\begin{equation}\label{PL.PS1.Re.mu-nu}
	m_0 \mu_i  \leq \nu_i.
\end{equation}
From \eqref{PL.PS1.mu-nu} and \eqref{PL.PS1.Re.mu-nu}, we deduce
\begin{equation*}
		S \min \left\{(m_0\mu_i)^{\frac{1}{p^\ast(x_i)}},(m_0\mu_i)^{\frac{1}{q^\ast(x_i)}}\right\}\leq \max \left\{\mu_i^{\frac{1}{p(x_i)}},\mu_i^{\frac{1}{q(x_i)}}\right\}.
\end{equation*}
Hence,
\begin{equation*}
	S (m_0\mu_i)^{\frac{1}{\eta_i}}\leq \mu_i^{\frac{1}{\xi_i}},
\end{equation*}
where $\eta_i\in \{p^\ast(x_i),q^\ast(x_i)\}$ and $\xi_i\in\{p(x_i),q(x_i)\}$. This infers
\begin{equation}\label{PL.PS1.mu_i.2}
	S^{\frac{\eta_i\xi_i}{\eta_i-\xi_i}} m_0^{\frac{\xi_i}{\eta_i-\xi_i}}\leq \mu_i.
\end{equation}
Note that 
\begin{equation*}
	\left(\frac{pq^*}{q^*-p}\right)^-\leq \frac{p(x_i)q^*(x_i)}{q^*(x_i)-p(x_i)}\leq \frac{\eta_i\xi_i}{\eta_i-\xi_i}\leq \frac{q(x_i)p^*(x_i)}{p^*(x_i)-q(x_i)}\leq \left(\frac{qp^*}{p^*-q}\right)^+
\end{equation*}
and
\begin{equation*}
	\left(\frac{p}{q^*-p}\right)^-\leq\frac{p(x_i)}{q^*(x_i)-p(x_i)}\leq \frac{\xi_i}{\eta_i-\xi_i}\leq \frac{q(x_i)}{p^*(x_i)-q(x_i)}\leq \left(\frac{q}{p^*-q}\right)^+.
\end{equation*}
Combining these estimates with \eqref{PL.PS1.mu_i.2} gives
\begin{equation}\label{PL.PS1.mu_i.3}
	\min\{S^{n_1},S^{n_2}\}\min\{m_0^{\tau_1},m_0^{\tau_2}\}\leq \mu_i,
\end{equation}
where $n_1:=\left(\frac{pq^*}{q^*-p}\right)^-$, $n_2:=\left(\frac{qp^*}{p^*-q}\right)^+$,  $\tau_1:=\left(\frac{p}{q^*-p}\right)^-$ and $\tau_2:=\left(\frac{q}{p^*-q}\right)^+$.

On the other hand, using \eqref{mu1} and \eqref{Le.PS1.PS-seq} again and repeating the argument leading to  \eqref{PL.PS1.E1} we have
\begin{align*}
	c+o_n(1)=&\Phi_\lambda(u_n)-\frac{1}{r_1^-}\left\langle\Phi_\lambda'(u_n) ,u_n\right\rangle\\ 
	\geq& 4K_0\int_\Omega \mathcal{H}(x,|\nabla u_n|) \diff x-\frac{\lambda }{r_1^-}\int_{\Omega}\big[C_4|u_n|^{\sigma(x)}+C_5|u_n|^{p(x)}\big]\diff x\\ 
	\geq& K_0\int_\Omega \mathcal{H}(x,|\nabla u_n|)\diff x+K_0\kappa_1\int_\Omega |u_n|^{p(x)}\diff x \\ 
	&+\left(K_0\kappa_1-\frac{\lambda C_5}{r_1^-}\right)\int_\Omega |u_n|^{p(x)} \diff x
	-\frac{\lambda C_4 }{r_1^-}\int_{\Omega} |u_n|^{\sigma(x)}\diff x\\
	\geq& K_0\int_\Omega \mathcal{H}(x,|\nabla u_n|)\diff x+K_0\kappa_1\int_\Omega |u_n|^{p(x)}\diff x 
	-\frac{\lambda C_4 }{r_1^-}\int_{\Omega} |u_n|^{\sigma(x)}\diff x.
\end{align*}
Passing to the limit as $n \to \infty$ in the last inequality, utilizing  \eqref{PL.PS1.mu}, Proposition~\ref{prop_S-C-E}, we have
\begin{equation*}
	c\geq K_0\mu_i+K_0\kappa_1\int_\Omega |u|^{p(x)}-\frac{\lambda C_4}{r_1^-}\int_\Omega |u|^{\sigma(x)}\diff x.
\end{equation*}	
By Proposition~\ref{prop.Holder}, we have
$$\int_{\Omega}|u|^{\sigma(x)}\diff x\leq 2\|1\|_{\frac{\ell(\cdot)}{\ell(\cdot)-1}}\big\||u|^{\sigma(\cdot)}\big\|_{\ell(\cdot)},$$
where $\ell(\cdot):=\frac{p(\cdot)}{\sigma(\cdot)}$. 
Then, it follows from the last two estimates that 
\begin{equation}\label{PL.PS1.c}
	c\geq\,  
	K_0\mu_i-\lambda b \big\||u|^{\sigma(\cdot)}\big\|_{\ell(\cdot)}+a \int_{\Omega}|u|^{p(x)}\diff x,
\end{equation}	
where $a:=	K_0\kappa_1$ and $b:= \frac{2C_4}{r_1^-}\|1\|_{\frac{\ell(\cdot)}{\ell(\cdot)-1}}$. By  Proposition~\ref{prop.nor-mod.D} we deduce from \eqref{PL.PS1.c} that
\begin{equation}\label{PL.PS1.c1}
	c\geq \min\left\{h_+\left(t_u\right),h_-\left(t_u\right) \right\},
\end{equation}
where   $t_u:=\big\||u|^{\sigma(\cdot)}\big\|_{\ell(\cdot)}$ and for $*\in\{+,-\}$, 
$$h_*(t):=K_0\mu_i-\lambda b t+a t^{\ell^*} \ \ \text{for}\ \ t\geq 0.$$
Note that for $*\in\{+,-\}$, it holds that
\begin{align*}
h_*\left(t_u\right)\geq \min_{t\geq 0}h_*(t)=h_*\left(\round{\frac{b\lambda}{a\ell^*}}^{\frac{1}{\ell^*-1}}\right)=K_0\mu_i-(\ell^*)^{-\frac{\ell^*}{\ell^*-1}} (\ell^*-1)a^{-\frac{1}{\ell^*-1}}b^{\frac{\ell^*}{\ell^*-1}}\lambda^{\frac{\ell^*}{\ell^*-1}}.
\end{align*}
Combining this with \eqref{PL.PS1.c1} gives
\begin{equation}\label{PL.PS1.c2}
	c\geq K_0\mu_i-K\max\left\{\lambda^{\frac{\ell^+}{\ell^+-1}},\lambda^{\frac{\ell^-}{\ell^--1}}\right\}.
\end{equation}
where
\begin{equation}\label{K}
	K:=\max_{*\in \{{+,-}\}}(\ell^*)^{-\frac{\ell^*}{\ell^*-1}} (\ell^*-1)a^{-\frac{1}{\ell^*-1}}b^{\frac{\ell^*}{\ell^*-1}}.
\end{equation}
Then, by taking into account \eqref{PL.PS1.mu_i.3} we derive from \eqref{PL.PS1.c2} that
$$c\geq K_0\min\{S^{n_1},S^{n_2}\}\min\{m_0^{\tau_1},m_0^{\tau_2}\}-K\max\left\{\lambda^{\frac{\ell^+}{\ell^+-1}},\lambda^{\frac{\ell^-}{\ell^--1}}\right\},$$
a contradiction with \eqref{Le.PS1.c}. That is, we have shown that  $I=\emptyset$, and thus, $\int_\Omega \mathcal{B}(x,u_n)\diff x \to \int_\Omega \mathcal{B}(x,u)\diff x$ as $n\to\infty$ by virtue of \eqref{PL.PS1.nu}. From this and \eqref{PL.PS1.a.e}, by invoking  Lemma~\ref{L.brezis-lieb} we obtain
 $$\int_\Omega \mathcal{B}(x,u_n-u)\diff x \to 0.$$
Equivalently, we derive
\begin{equation} \label{un.conv.B}	
u_n\to u\ \ \text{in}\ \ L^{\mathcal{B}}(\Omega)
	\end{equation}
in view of Proposition~\ref{prop.nor-mod.D}. Moreover, from \eqref{Diff1} we have
\begin{align}\label{PS1.S+.1}
\notag	m_0\bigg|\int_\Omega &A(x,\nabla u_n) \cdot \nabla(u_n-u)\diff x\bigg|\\
\notag	\leq& M_0\left(\int_\Omega\mathcal{A}(x,\nabla u)\,\diff x\right)\bigg|\int_\Omega A(x,\nabla u_n) \cdot \nabla(u_n-u)\diff x\bigg|\\
	\notag	\leq& \left|\langle \Phi_\lambda' (u_n),u_n-u\rangle\right|+\lambda\int_\Omega \left|f(x,u_n)(u_n-u)\right|\diff x\\
	&\quad\quad+\int_\Omega \left[c_1(x)|u_n|^{r_1(x)-1}|u_n-u|+c_2(x)a(x)^{\frac{r_2(x)}{q(x)}}|u_n|^{r_2(x)-1}|u_n-u|\right]\,\diff x.
\end{align}
By \eqref{Le.PS1.PS-seq} and the boundedness of $\{u_n\}_{n\in\N}$ in $X$ we obtain
\begin{equation}\label{PS1.S+.2}
	\lim_{n\to\infty}\left|\langle \Phi_\lambda' (u_n),u_n-u\rangle\right|=0.
\end{equation}
On the other hand, by means of $(\mathcal{F}_0)$ and Proposition~\ref{prop.Holder} we deduce 
\begin{align*}
\int_\Omega \big|f(x,u_n)(u_n-u)\big|\,\diff x
 &\leq C_1\int_\Omega\round{1+|u_n|^{\alpha(x)-1}}|u_n-u|\,\diff x \notag\\
&\leq C_1\|u_n-u\|_1+2C_1\big\||u_n|^{\alpha(\cdot)-1}\big\|_{\frac{\alpha(\cdot)}{\alpha(\cdot)-1}}\|u_n-u\|_{\alpha(\cdot)}.
\end{align*}
From this and \eqref{PL.PS1.w-conv}, by virtue of Proposition~ \ref{prop_embs} we obtain
\begin{equation}\label{PS1.S+.3}
	\lim_{n\to\infty}\int_\Omega \big|f(x,u_n)(u_n-u)\big|\,\diff x=0.
\end{equation}
Also, by applying Proposition~\ref{prop.Holder} we have
\begin{align*}
	\notag
	\int_\Omega c_1(x)|u_n|^{r_1(x)-1}|u_n-u|\,\diff x\leq 2\big\||u_n|^{r_1(\cdot)-1}\big\|_{L^{\frac{r_1(\cdot)}{r_1(\cdot)-1}}\big(c_1,\Omega\big)}\|u_n-u\|_{L^{r_1(\cdot)}\big(c_1,\Omega\big)}
\end{align*}
and
\begin{align*}
		\int_\Omega c_2(x)a(x)^{\frac{r_2(x)}{q(x)}}|u_n|^{r_2(x)-1}|u_n-u|\,\diff x\leq 2\big\||u_n|^{r_2(\cdot)-1}\big\|_{L^{\frac{r_2(\cdot)}{r_2(\cdot)-1}}\big(a^{\frac{r_2}{q}},\Omega\big)}\|u_n-u\|_{L^{r_2(\cdot)}\big(a^{\frac{r_2}{q}},\Omega\big)}.
\end{align*}
From the last two estimates and \eqref{un.conv.B} we arrive at
\begin{equation}\label{PS1.S+.4}
	\lim_{n\to\infty}\int_\Omega \left[c_1(x)|u_n|^{r_1(x)-1}+c_2(x)a(x)^{\frac{r_2(x)}{q(x)}}|u_n|^{r_2(x)-1}\right]|u_n-u|\,\diff x=0.
\end{equation}
By exploiting \eqref{PS1.S+.2}-\eqref{PS1.S+.4}, we derive from \eqref{PS1.S+.1} that
\begin{equation*}
	\lim_{n\to\infty}\int_\Omega A(x,\nabla u_n) \cdot \nabla(u_n-u)\diff x=0.
\end{equation*}
This fact implies that $u_n\to u$ in $X$ %by the $(\textup{S}_+)$-property of the operator $A$, see 
in view of \cite[Theorem 3.3]{CGHW}, and the proof is complete.
\end{proof}
By Propositions~\ref{prop_embs} and~\ref{prop_S-C-E},  we can take a constant $C_8>1$ such that
\begin{equation}\label{norms-norm}
	\max\left\{\|u\|_{\sigma(\cdot)},\|u\|_{\mathcal{B}}\right\}\leq C_8\|u\|,\quad \forall \, u\in X.
\end{equation}
Set 
\begin{equation}\label{lambda2-k0}
	\lambda_2:=\frac{m_0 \kappa_1}{2q^+C_3},
\end{equation}
where $C_3$ is given by $(\mathcal{F}_2)$. Let $\lambda\in(0,\lambda_2)$. By $(\mathcal{F}_2)$, \eqref{mu1} and \eqref{Est.hat.M0}, we have
\begin{align*}
	\notag \Phi_\lambda(u)\geq& \frac{m_0}{q^+}\int_\Omega \mathcal{H}(x,|\nabla u|)\diff x-\lambda C_2\int_\Omega |u|^{\sigma(x)}\diff x - \lambda C_3\int_\Omega |u|^{p(x)}\diff x - \frac{1}{r_1^-}\int_\Omega \mathcal{B}(x,u) \diff x\\
	\geq& \frac{m_0}{2q^+}\int_\Omega \mathcal{H}(x,|\nabla u|) \diff x + \round{\frac{m_0\kappa_1}{2q^+} - \lambda C_3}\int_\Omega |u|^{p(x)}\diff x-\lambda C_2\int_\Omega |u|^{\sigma(x)}\diff x\\
	&-\frac{1}{r_1^-}\int_\Omega \mathcal{B}(x,u) \diff x,~ \forall u\in X.
\end{align*}
By means of Proposition~\ref{prop.nor-mod.D} and \eqref{norms-norm}, we deduce from the last inequality that for all $u\in X$ with $\|u\|\leq 1$,
\begin{align*}
	\notag\Phi_\lambda(u)&\geq \frac{m_0}{2q^+}\min\left\{\|u\|^{p^-},\|u\|^{q^+}\right\}-\lambda C_2\max\left\{|u|_{\sigma(\cdot)}^{\sigma^+},|u|_{\sigma(\cdot)}^{\sigma^-}\right\} 
	-\frac{1}{r_1^-}\max\left\{\|u\|_{\mathcal{B}}^{r_1^-},\|u\|_{\mathcal{B}}^{r_2^+}\right\}\\
	\notag&\geq \frac{m_0}{2q^+}\min\left\{\|u\|^{p^-},\|u\|^{q^+}\right\}-\lambda C_2 C_8^{\sigma^+}\max\left\{\|u\|^{\sigma^-},\|u\|^{\sigma^+}\right\}-\frac{C_8^{r_2^+}}{r_1^-}\max\left\{\|u\|^{r_1^-},\|u\|^{r_2^+}\right\}\\
	&\geq \frac{m_0}{2q^+}\|u\|^{q^+}-\lambda C_2 C_8^{\sigma^+}\|u\|^{\sigma^-}-\frac{C_8^{r_2^+}}{r_1^-}\|u\|^{r_1^-}. 
\end{align*}
That is,
\begin{equation}\label{PTcc.gi}
	\Phi_\lambda(u)\geq  g_\lambda(\|u\|)\ \ \text{for} \ \ \|u\|\leq1,
\end{equation} 
where $g_\lambda\in C[0,+\infty)$ is given by
$$g_\lambda(t):=\frac{m_0}{2q^+}t^{q^+}-\lambda C_2 C_8^{\sigma^+}t^{\sigma^-}-\frac{C_8^{r_2^+}}{r_1^-}t^{r_1^-},\quad t\ge 0.$$
In order to analyze the behavior of $g_\lambda$, rewrite
$$g_\lambda(t)=C_2 C_8^{\sigma^+}t^{\sigma^-}\left(h(t)-\lambda\right),$$
where
$$h(t):=a_0t^{q^+-\sigma^-}-b_0t^{r_1^--\sigma^-},$$ 
with $a_0:=m_0(2q^+C_2C_8^{\sigma^+})^{-1}>0$ and $b_0:=C_8^{r_2^+-\sigma^+}(C_2r_1^-)^{-1}>0$. Clearly, 
\begin{align}\label{lambda3}
	\notag\lambda_3:&=\max_{t\geq 0}\, h(t)=h\left(\left[\frac{(q^+-\sigma^-)a_0}{(r_1^--\sigma^-)b_0}\right]^{\frac{1}{r_1^--q^+}}\right)\\
	&=a_0^{\frac{r_1^--\sigma^-}{r_1^--q^+}}b_0^{\frac{\sigma^--q^+}{r_1^--q^+}}\left(\frac{q^+-\sigma^-}{r_1^--\sigma^-}\right)^{\frac{q^+-\sigma^-}{r_1^--q^+}}\frac{r_1^--q^+}{r_1^--\sigma^-}>0.
\end{align}
Clearly, for any $\lambda\in (0,\lambda_3)$, $g_\lambda(t)$ has exactly two positive roots $t_1(\lambda)$ and $t_2(\lambda)$ with 
\begin{equation}\label{t*}
	0<t_1(\lambda)<\left[\frac{(q^+-\sigma^-)a_0}{(r_1^--\sigma^-)b_0}\right]^{\frac{1}{r_1^--q^+}}=:t_*<t_2(\lambda).
\end{equation}
Moreover, it holds that
\begin{equation}\label{g-la}
	g_\lambda(t)\begin{cases}
		<0,\quad t\in (0,t_1(\lambda))\cup (t_2(\lambda),+\infty),\\
		>0,\quad t\in (t_1(\lambda),t_2(\lambda))
	\end{cases}
\end{equation}
and
\begin{equation}\label{lim.t1}
	\lim_{\lambda\to 0^+}t_1(\lambda)=0.
\end{equation}
By \eqref{lim.t1}, we find $\lambda_4>0$ such that
\begin{equation}\label{t_1}
	t_1(\lambda)<\min\left\{(2q^+)^{-1/p^-},\left((2q^+)^{-1}t_*^{q^+}\right)^{1/p^-},t_0^{1/p^-}\right\},\quad \forall \lambda\in (0,\lambda_4),
\end{equation}
where $t_0$ and $t_*$ are given in \eqref{t0-sub} and \eqref{t*}, respectively.
Set 
\begin{equation}\label{la*1}
	\lambda_{\ast,1}:=\min_{i\in \{1,2,3,4\}}\lambda_i
\end{equation} 
with $\lambda_1$, $\lambda_2$, $\lambda_3$ and $\lambda_4$ given in \eqref{lambda1}, \eqref{lambda2-k0}, \eqref{lambda3} and \eqref{t_1}, respectively. For each $\lambda\in \left(0,\lambda_{\ast,1}\right)$, define a truncated functional $T_\lambda: X\to\R$ as
\begin{align*}\label{T_lambda}
	\notag T_\lambda(u):=&\widehat{M}_0\left(\int_\Omega \mathcal{A}(x,\nabla u) \diff x\right)-\phi\left(\int_\Omega \mathcal{A}(x,\nabla u) \diff x\right)\left[\lambda\int_\Omega F(x,u) 
	\diff x+\int_\Omega \widehat{B}(x,u) \diff x\right],\quad u\in X,
\end{align*}
where $\phi\in C_c^\infty(\R)$ satisfies $0\leq \phi(\cdot)\leq 1$, $\phi(t)=1$ for $|t|\leq t_1(\lambda)^{p^-}$ and $\phi(t)=0$ for $|t|\geq 2t_1(\lambda)^{p^-}$. Clearly, $T_\lambda\in C^1(X,\R)$ and it holds that
\begin{equation}\label{T_lambda.Est1}
	T_\lambda(u)\geq \Phi_\lambda(u),\quad \forall u\in X,
\end{equation}
\begin{equation}\label{T_lambda.Est2}
	T_\lambda(u)=\Phi_\lambda(u) \ \ \text{for any} \ \ u\in X \ \ \text{with} \ \ \int_\Omega \mathcal{A}(x,\nabla u) \diff x<t_1(\lambda)^{p^-},
\end{equation}
and
\begin{equation}\label{T_lambda.Est3}
	T_\lambda(u)=\widehat{M}_0\left(\int_\Omega \mathcal{A}(x,\nabla u) \diff x\right)\ \ \text{for any} \ \ u\in X \ \ \text{with} \ \ \int_\Omega \mathcal{A}(x,\nabla u) \diff x>2t_1(\lambda)^{p^-}.
\end{equation}

\begin{lemma}\label{T_lambda(u)<0}
	Let $\lambda\in \left(0,\lambda_{\ast,1}\right)$. Then, for $u\in X$ with $T_\lambda(u)<0$, it holds that $\int_\Omega \mathcal{A}(x,\nabla u) \diff x<t_1(\lambda)^{p^-}$; in particular, $T_\lambda(u)=\Phi_\lambda(u)$ and $T_\lambda'(u)=\Phi_\lambda'(u)$.
\end{lemma}
\begin{proof}
	Let $\lambda\in \left(0,\lambda_{\ast,1}\right)$ and let $u\in X$ with $T_\lambda(u)<0$. Then, we have $\Phi_\lambda(u)<0$ due to \eqref{T_lambda.Est1}. We claim that $\|u\|< 1$. Indeed, assume on the contrary that $\|u\|\geq 1$, then $\int_\Omega \mathcal{H}(x,|\nabla u|)\diff x\geq 1$ in view of Proposition~\ref{prop.nor-mod.D}. Thus, $\int_\Omega \mathcal{A}(x,\nabla u) \diff x\geq \frac{1}{q^+}\int_\Omega \mathcal{H}(x,|\nabla u|)\diff x\geq\frac{1}{q^+}>2t_1(\lambda)^{p^-}$ due to \eqref{t_1}. By combining this with \eqref{Est.hat.M0} and \eqref{T_lambda.Est3} we derive
	\begin{equation*}
		m_0\int_\Omega \mathcal{A}(x,\nabla u) \diff x\leq \widehat{M}_0\left(\int_\Omega \mathcal{A}(x,\nabla u) \diff x\right)=T_\lambda(u)<0,
	\end{equation*}
	a contradiction. That is to say, it must be that $\|u\|<1$; hence,  $g_\lambda(\|u\|)\leq \Phi_\lambda(u)<0$ due to \eqref{PTcc.gi}. Then by \eqref{g-la}, either $\|u\|<t_1(\lambda)$ or $\|u\|>t_2(\lambda)>t_*$. The latter case implies that $\int_\Omega \mathcal{A}(x,\nabla u) \diff x\geq \frac{1}{q^+}\int_\Omega \mathcal{H}(x,|\nabla u|)\diff x\geq\frac{1}{q^+}\|u\|^{q^+}>\frac{t_*^{q^+}}{q^+}>2t_1(\lambda)^{p^-}$ in view of Proposition~\ref{prop.nor-mod.D} and \eqref{t_1}; hence, by \eqref{T_lambda.Est3} we get that $T_\lambda(u)\geq 0$, a contradiction. Thus, it must hold that $\|u\|<t_1(\lambda)$, and therefore, $\int_\Omega \mathcal{A}(x,\nabla u) \diff x\leq \frac{1}{p^-}\int_\Omega \mathcal{H}(x,|\nabla u|)\diff x\leq\frac{\|u\|^{p^-}}{p^-}<\frac{t_1(\lambda)^{p^-}}{p^-}$. The proof is complete.
\end{proof}

In order to complete the proof of Theorem~\ref{Theo.cc}, we will construct a sequence $\{c_k\}_{k\in\N}$ of negative critical values of $T_\lambda$ via the genus theory. As a result of Lemma~\ref{T_lambda(u)<0}, critical points of $T_\lambda$ corresponding to these $c_k$ are also solutions to \eqref{e1.1'} with $\theta=1$. Denote by $\Sigma$ the set of all closed subset $E\subset X\setminus\{0\}$ such that $E=-E$, namely, $u\in E$ implies $-u\in E$. For $E\in\Sigma$, denote by $\gamma(E)$ the genus of a $E$ (see \cite{AR,Rab.Mo86} for the definition and properties of the genus). For each $\tau\in\R$, define
  $$T_\lambda^{\tau}:=\{u \in X: \, T_\lambda(u) \leq \tau\}.$$
  By $(\mathcal{F}_1)$ and the definition of $T_\lambda$, it holds that $T_\lambda^{-\epsilon}\in\Sigma$ for all $\epsilon>0$, and we have the following.
\begin{lemma}\label{genus.k}
Let $\lambda\in \left(0,\lambda_{\ast,1}\right)$. Then, for each $k \in \mathbb{N}$, there exists $\epsilon>0$ such that
	$$
	\gamma(T_\lambda^{-\epsilon}) \geq k.
	$$ 
\end{lemma}
\begin{proof}
	
Let $B$ be as in $(\mathcal{F}_2)$ and denote
$$p_B^-:=\inf_{x \in B}p(x)\ \text{ and } \ \sigma^+_{B}:=\sup_{x \in B}\sigma(x).$$
We construct a sequence $\{X_k\}_{k\in\N}$ of linear subspaces of $X$ as follows. For each $k\in\mathbb{N}$, define
\begin{equation}\label{Xk}
	X_k:=\operatorname{span}\{\varphi_1,\varphi_2,\cdots,\varphi_k\},
\end{equation}
where $\varphi_k$ is an  eigenfunction corresponding to the $k^{\text{th}}$ eigenvalue of the following eigenvalue problem : 
\begin{equation*}
	\begin{cases}
		-\Delta u=\mu u \quad &\text{in } B,\\
		u=0 \quad &\text{on } \partial B
	\end{cases}
\end{equation*}
and is extended on $\Omega$ by putting $\varphi_k(x)=0$ for $x\in\Omega\setminus B.$ Then $X_{k}$ is a linear subspace of $X$ with dimension $k$.
Since all norms on $X_k$ are mutually equivalent, we find $\delta_k>t_1(\lambda)^{-1}(>1)$ such that	
\begin{equation}\label{Equi.norm}
	\delta_k^{-1}\|u\|_{L^{\sigma(\cdot)}(B)}\leq \|u\|\leq \delta_k\|u\|_{L^{\sigma(\cdot)}(B)},\quad \forall u\in X_k.
\end{equation}
For any $u\in X_k$ with $\|u\|<\delta_k^{-1}<1$,  we have $$\int_\Omega \mathcal{A}(x,\nabla u) \diff x\leq \frac{\|u\|^{p^-}}{p^-}<t_1(\lambda)^{p^-} \text{ and } \|u\|_{L^{\sigma(\cdot)}(B)}<1.$$
Combining this with \eqref{Est.hat.M0}, \eqref{T_lambda.Est2}, \eqref{Equi.norm}, $(\mathcal{F}_2) \textnormal{(ii)}$ and invoking Proposition~\ref{prop.nor-mod.D} we obtain
\begin{align*}
	T_\lambda(u)=\Phi_\lambda(u) &\leq \frac{M(t_0)}{p^-}\int_B \mathcal{H}(x,|\nabla u|)\diff x-\lambda\int_B F(x,u)\diff x\\
	&\leq \frac{M(t_0)}{p^-}\|u\|^{p_B^-}-C_6\lambda\int_B|u|^{\sigma(x)}\diff x \\
	&\leq \frac{M(t_0)}{p^-}\|u\|^{p_B^-}-C_6\lambda \|u\|^{\sigma_B^+}_{L^{\sigma(\cdot)}(B)}\\
	&\leq \frac{M(t_0)}{p^-}\|u\|^{p_B^-}-C_6\lambda \delta_k^{-\sigma_B^+}\|u\|^{\sigma_B^+}.
\end{align*}	
Thus
\begin{align*}
	T_\lambda(u)\leq \frac{M(t_0)}{p^-}\|u\|^{\sigma_B^+}\left(\|u\|^{p_B^--\sigma_B^+}-\frac{p^- C_6 \lambda\delta_k^{-\sigma_B^+}}{M(t_0)}\right).
\end{align*}
By taking  $r$ and $\epsilon$ with 
$$0<r<\min\left\{\delta_k^{-1},\left(\frac{p^- C_6 \lambda \delta_k^{-\sigma_B^+}}{M(t_0)}\right)^{\frac{1}{p_B^--\sigma_B^+}}\right\},$$
and 
$$\epsilon:=-\frac{M(t_0)}{p^-}r^{\sigma_B^+}\left(r^{p_B^--\sigma_B^+}-\frac{p^- C_6 \lambda\delta_k^{-\sigma_B^+}}{M(t_0)}\right),$$
we have
$$
T_\lambda(u)\leq-\epsilon < 0,\quad \forall u\in {S_r}:=\{u\in X_k:\, \|u\|=r \}.
$$
It follows that $S_r\subset T_\lambda^{-\epsilon}$ and thus,
$$\gamma( T_\lambda^{-\epsilon})\geq \gamma(S_r)=k.
$$
The proof is complete.	
\end{proof}

Let $\lambda\in \left(0,\lambda_{\ast,1}\right)$. Define
\begin{gather*}
	\Sigma_k:=\{E\in\Sigma: \,   \gamma(E) \geq k\}
\end{gather*}
and 
\begin{equation}\label{ck}
c_{k}:= \inf_{E\in \Sigma_k} \sup_{u \in
	E}T_\lambda(u).
\end{equation}
We have the following.
\begin{lemma}\label{Le.c_k<0}
	For each $k \in \mathbb{N}$, it holds that $-\infty<c_k<0.$
\end{lemma}
\begin{proof} For each $k \in \mathbb{N}$, let us take $\epsilon$ as in Lemma~\ref{genus.k}. Clearly, $T_\lambda^{-\epsilon}\in \Sigma_k$ by means of Lemma~\ref{genus.k}. Moreover, from the embedding results in Propositions~\ref{prop_embs} and \ref{prop_S-C-E}, it is not difficult to deduce from the definition of $T_\lambda$ that $T_\lambda$ is bounded below. Thus, we infer
	$$-\infty<c_k\leq \sup_{u \in
		T_\lambda^{-\epsilon}}T_\lambda(u)<-\epsilon<0.$$
	The proof is complete.
\end{proof}
Now, we will choose $\lambda_{\ast,2}>0$ such that
\begin{equation*}
	K_0\min\{S^{n_1},S^{n_2}\}\min\{m_0^{\tau_1},m_0^{\tau_2}\}-K\max\left\{\lambda_{\ast,2}^{\frac{\ell^+}{\ell^+-1}},\lambda_{\ast,2}^{\frac{\ell^-}{\ell^--1}}\right\}>0,
\end{equation*}
where $K_0,K,n_1,n_2,\tau_1,\tau_2$ and $\ell(\cdot)$ are determined in Lemma~\ref{Le.PS1}.
 Set 
\begin{equation}\label{lambda_*}
	\lambda_\ast:=\min\left\{\lambda_{\ast,1},\lambda_{\ast,2}\right\},
\end{equation}
where $\lambda_{\ast,1}$ is given by \eqref{la*1}. By employing the deformation lemma we can obtain the following.

\begin{lemma}\label{Le.Seq.sol}
	For any $\lambda\in\left(0,\lambda_\ast\right)$ and $k\in\N$, if $c=c_{k}=c_{k+1}=\dots =c_{k+m}$ for some $m \in \mathbb{N}$, then $K_c\in\Sigma$ and 
	$$ \gamma(K_{c})\geq m+1, $$
	where $K_{c}:=\{u \in X\backslash\{0\}: T_\lambda'(u)=0  \text{ and }
	T_\lambda(u)=c\}$.
\end{lemma}

\begin{proof}
	Let $\lambda\in\left(0,\lambda_\ast\right)$. Then,  Lemma~\ref{Le.c_k<0} and the choice of $\lambda_\ast$ imply
	$$c<0<K_0\min\{S^{n_1},S^{n_2}\}\min\{m_0^{\tau_1},m_0^{\tau_2}\}-K\max\left\{\lambda^{\frac{\ell^+}{\ell^+-1}},\lambda^{\frac{\ell^-}{\ell^--1}}\right\}.$$
	Thus, $K_c\in\Sigma$ and $K_{c}$ is a compact set by means of Lemmas~\ref{Le.PS1} and \ref{T_lambda(u)<0}. Although the conclusion of the lemma can be easily obtained by a standard argument using the deformation lemma (see e.g., \cite[Lemma 1.3]{AR}), we will present the details here for seeking the completeness. Assume on the contrary that $\gamma(K_{c})\leq m$. Then, there exists a closed
	and symmetric neighborhood $U$ of $ K_{c}$ such that
	$\gamma(U)= \gamma(K_{c}) \leq m$. Note that we can choose
	$U\subset T_{\lambda}^{0}$ because $c<0$. By the deformation lemma  \cite[Lemma 1.3]{AR}, we have an odd homeomorphism $\eta$ of $X$ onto $X$ such that $\eta(T_{\lambda}^{c+\delta}\setminus U)\subset
	T_{\lambda}^{c-\delta}$ for some $\delta > 0$ with $0<\delta < -c$.
	Thus, $T_{\lambda}^{c+\delta}\subset T_{\lambda}^{0}$ and by
	definition of $c=c_{k+m}$, there exists $E \in \Sigma_{k+m}$ such
	that $ \sup_{u \in E} T_\lambda(u)< c+\delta$, that is,
	$E \subset T_{\lambda}^{c+\delta}$ and
	\begin{equation}\label{estrela1}
		\eta(E\setminus U) \subset \eta ( T_{\lambda}^{c+\delta}\setminus U)\subset
		T_{\lambda}^{c-\delta}.
	\end{equation}
	But $\gamma(\overline{E\setminus U})\geq \gamma(E)-\gamma(U) \geq k$ and
	$\gamma(\eta(\overline{E\setminus U}))\geq  \gamma(\overline{E \setminus U})\geq k$.
	Then $\eta(\overline{E\setminus U}) \in \Sigma_{k}$ and hence, by \eqref{ck} and \eqref{estrela1} we arrive at
	\begin{equation*}
		c_{k}\leq  \sup_{u \in
			\eta(\overline{E\setminus U})}T_\lambda(u)\leq c-\delta=c_k-\delta,
	\end{equation*}
	a contradiction. Hence, $\gamma(K_{c})>m$ and the lemma is proved.
\end{proof}
\begin{proof}[\textbf{Proof of Theorem \ref{Theo.cc} }]
	Let $\lambda\in (0,\lambda_*)$ with $\lambda_*$ given in \eqref{lambda_*}. By Lemmas~\ref{Le.c_k<0} and \ref{Le.Seq.sol}, $T_\lambda$ admits a sequence $\{u_k\}_{k\in\N}$ of critical points with $T_\lambda(u_k)<0$ for all $k\in\N$. By Lemma~\ref{T_lambda(u)<0} and \eqref{t_1}, $\{u_k\}_{k\in\N}$ are also critical points of $\Phi_\lambda$ with 
	$\int_\Omega \mathcal{A}(x,\nabla u_k) \diff x<t_1(\lambda)^{p^-}<t_0$; hence, $\{u_k\}_{k\in\N}$ are solutions to problem~\eqref{e1.1'}. Now, denote by $u_\lambda$ one of those $u_k$. By Lemma~\ref{T_lambda(u)<0} again, we have
	$$\frac{1}{q^+}\int_\Omega \mathcal{H}(x,|\nabla u_\lambda|) \diff x\leq \int_\Omega \mathcal{A}(x,\nabla u_\lambda) \diff x<t_1(\lambda)^{p^-}.$$
	Combining this with \eqref{lim.t1}  derives
	$$\lim_{\lambda\to 0^+}\int_\Omega \mathcal{H}(x,|\nabla u_\lambda|) \diff x=0,$$
	and this is equivalent to
	$$\lim_{\lambda\to 0^+}\|u_\lambda\| =0$$
	in view of Proposition~\ref{prop.nor-mod.D}. The proof is complete.
\end{proof}

\subsection{The generalized superlinear case}${}$

In this subsection, we will prove Theorem~\ref{Theo.sl.M=1} by developing the idea of \cite[Theorem 2.2]{KK16}, which was used only for the $(p,q)$-Laplacian involving subcritical growth. More precisely, we will make use of the genus theory again to determine the critical points of the energy functional associated with problem~\eqref{e1.1'}. The assumption $M(\cdot) \equiv 1$  enables us to work directly with the energy functional instead of its truncation. Now, let us assume that all assumptions imposed in Theorem~\ref{Theo.sl.M=1} hold and let $M(\cdot) \equiv 1$ and $\lambda=1$. 

Define the energy functional $\Psi_\theta: X\to\R$ as
$$\Psi_\theta(u):=\int_\Omega \mathcal{A}(x,\nabla u)\,\diff x-\int_\Omega F(x,u) \, \diff x-\theta \int_\Omega \widehat{B}(x,u)\,\diff x,\quad u\in X,$$
where $\widehat{B}$ is given by \eqref{B.hat}. As observed in the previous subsection, $\Psi_\theta$ is of $C^1$ and its Fr\'echet derivative $\Psi^\prime_\theta: X\to X^\ast$ is given by 
\begin{equation}\label{Diff2}
	\left\langle \Psi^\prime_\theta (u),v\right\rangle=  \int_\Omega A(x,\nabla u)\cdot\nabla v\,\diff x-\int_\Omega f(x,u)v\,\diff x-\theta\int_\Omega B(x,u)v\,\diff x,\quad \forall\, u,v\in X.
\end{equation}
Clearly, a critical point of $\Psi_\theta$ is a solution to \eqref{e1.1'} with $M(\cdot) \equiv 1$ and $\lambda=1$. Moreover, $\Psi_\theta$ is even and $\Psi_\theta(0)=0$. The following lemma is essential to verify the (local) $\textup{(PS)}$ condition for $\Psi_\theta$. 
\begin{lemma}\label{Le.PS4.c}
	For each  $\theta>0$, $\Psi_\theta$ satisfies the $\textup{(PS)}_c$ condition with $c\in \mathbb{R}$ such that
	\begin{equation}\label{PS2-c}
		c<\left(\frac{1}{\beta} - \frac{1}{r_1^-}\right)\min\{S^{n_1},S^{n_2}\}\min\{\theta^{-\tau_1},\theta^{-\tau_2}\}-\frac{\|e\|_1}{\beta},
	\end{equation}
	where $n_1$, $n_2$, $\tau_1$ and $\tau_2$ are given in Lemma~\ref{Le.PS1}.
\end{lemma}
\begin{proof}
Let $\theta>0$ and let $\{u_n\}_{n\in\N} \subset X$ be a $\text{(PS)}_c$ sequence for $\Psi_\theta$ with $c$ satisfying \eqref{PS2-c}, namely,
	\begin{equation}\label{ps2.e1}
		\Psi_\theta(u_n) \to c~ \text{ and } ~ \Psi'_\theta(u_n) \to 0~ \text{ as } ~ n\to\infty.
	\end{equation}
We first claim that $\{u_n\}_{n\in\N}$ is bounded in $X$. Indeed, by $(\mathcal{F}_3)$ and \eqref{mu1} it follows that for all $n\in \mathbb{N},$
\begin{align}\label{PT.sup.b0}
\notag	\Psi_\theta(u_n) - \frac{1}{\beta}\langle \Psi_\theta'(u_n), u_n\rangle & \geq \left(\frac{1}{q^+}-\frac{1}{\beta}\right) \int_\Omega\mathcal{H}(x,|\nabla u_n|) \diff x+\frac{1}{\beta}\int_\Omega \left[f(x,u_n)u_n - \beta F(x,u_n)\right]\, \diff x\\
	\notag & \hspace{1cm}  + \theta \left(\frac{1}{\beta} - \frac{1}{r_1^-}\right)\int_{\Omega}\mathcal{B}(x,u_n)\diff x\\
\notag	& \geq \left(\frac{1}{q^+}-\frac{1}{\beta}\right)
	\int_\Omega |\nabla u_n|^{p(x)}\,\diff x-\frac{\left(\beta-q^+\right)\kappa_1}{\beta q^+}
	\int_\Omega |u_n|^{p(x)}\,\diff x-\frac{\|e\|_1}{\beta}\\
\notag	& \hspace{1cm}+ \theta \left(\frac{1}{\beta} - \frac{1}{r_1^-}\right)\int_{\Omega}\mathcal{B}(x,u_n)\diff x\\
	&\geq \theta \left(\frac{1}{\beta} - \frac{1}{r_1^-}\right)\int_{\Omega}\mathcal{B}(x,u_n)\diff x-\frac{\|e\|_1}{\beta}.
\end{align}	
From \eqref{PT.sup.b0} and \eqref{ps2.e1}, we have that for all $n\in \mathbb{N}$ large,
\begin{equation}\label{PT.sup.b1}
	\int_{\Omega}\mathcal{B}(x,u_n)\diff x\leq \hat{C}(1+\|u_n\|),\ \ \forall n\in\mathbb{N},
\end{equation}		
where $\hat{C}$ is a positive constant independent of $u_n$. On the other hand, by \eqref{F} and invoking Proposition~\ref{prop.nor-mod.D} we have for $n$ large,
\begin{align}\label{PT.sup.b2}
	\notag\frac{1}{q^+}\left(\|u_n\|^{p^-}-1\right)&\leq \Psi_\theta(u_n)+\int_\Omega F(x,u_n)\diff x+\frac{\theta}{r_1^-}\int_{\Omega}\mathcal{B}(x,u_n)\diff x\\
	&\leq c+1+2C_1\int_{\Omega}\left(1+|u_n|^{\alpha(x)}\right)\diff x+\frac{\theta}{r_1^-}\int_{\Omega}\mathcal{B}(x,u_n)\diff x.
\end{align}
By Young's inequality we have
$$|u_n|^{\alpha(x)}\leq \frac{r_1(x)-\alpha(x)}{r_1(x)}c_1(x)^{-\frac{\alpha(x)}{r_1(x)-\alpha(x)}}+\frac{\alpha(x)}{r_1(x)}c_1(x)|u_n|^{r_1(x)}.$$
Plugging this into \eqref{PT.sup.b2} and noticing $c_1^{-\frac{\alpha}{r_1-\alpha}}\in L^1(\Omega)$ we arrive at
\begin{equation}\label{PT.PS1.un-b.2}
	\|u_n\|^{p^-}\leq \widetilde{C}\left(1+\int_{\Omega}\mathcal{B}(x,u_n)\diff x\right),
\end{equation}
	where $\widetilde{C}$ is a positive constant independent of $u_n$. From \eqref{PT.sup.b1} and \eqref{PT.PS1.un-b.2} we deduce the boundedness of $\{u_n\}_{n\in\N}$ in $X$ due to $p^->1$.  In view of  Theorem~\ref{Theo.ccp}, up to a subsequence we have
	\begin{eqnarray}
		u_n(x) &\to& u(x)  \quad \text{a.a.} \ \ x\in\Omega,\label{PL.PS2.a.e}\\
		u_n &\rightharpoonup& u  \quad \text{in} \  X,\label{PL.PS2.w-conv}\\
		\mathcal{H}(\cdot,|\nabla u_n|) &\overset{\ast }{\rightharpoonup }&\mu \geq \mathcal{H}(\cdot,|\nabla u|) + \sum_{i\in I} \mu_i \delta_{x_i} \ \text{in}\  \mathcal{M}(\overline{\Omega}),\label{PL.PS2.mu}\\
		\mathcal{B}(\cdot,u_n)&\overset{\ast }{\rightharpoonup }&\nu=\mathcal{B}(\cdot,u) + \sum_{i\in I}\nu_i\delta_{x_i} \ \text{in}\ \mathcal{M}(\overline{\Omega}),\label{PL.PS2.nu}\\
		S \min \left\{\nu_i^{\frac{1}{p^\ast(x_i)}},\nu_i^{\frac{1}{q^\ast(x_i)}}  \right\}&\leq& \max \left\{\mu_i^{\frac{1}{p(x_i)}},\mu_i^{\frac{1}{q(x_i)}}\right\}, \quad \forall i\in I.\label{PL.PS2.mu-nu}
	\end{eqnarray}
	As before, we will show that $I = \emptyset$. Indeed, by contradiction, we suppose  there exists $i\in I$. By repeating arguments leading to \eqref{PL.PS1.Re.mu-nu}, we obtain 
	\begin{equation*}
		\mu_i \leq \theta\nu_i.
	\end{equation*}
Combining this with \eqref{PL.PS2.mu-nu} and arguing as those leading to \eqref{PL.PS1.mu_i.3} we arrive at
\begin{equation}\label{est.mu_nu2}
	\min\{S^{n_1},S^{n_2}\}\min\{\theta^{-\tau_1},\theta^{-\tau_2}\}\leq \mu_i\leq \theta\nu_i,
\end{equation}
where $n_1$, $n_2$, $\tau_1$ and $\tau_2$ are given in Lemma~\ref{Le.PS1}.	On the other hand, by \eqref{PT.sup.b0} we have
	\begin{align}\label{e5.60}
		\notag c + o_n(1) & = \Psi_\theta(u_n) - \frac{1}{\beta}\langle \Psi_\theta'(u_n), u_n \rangle  \geq \theta \left(\frac{1}{\beta} - \frac{1}{r_1^-}\right)\int_{\Omega}\mathcal{B}(x,u_n)\diff x-\frac{\|e\|_1}{\beta}. 
	\end{align}	
	Passing to the limit as $n\to \infty$ in the last inequality and utilizing  \eqref{PL.PS2.nu} and \eqref{est.mu_nu2} we obtain
	$$
	c\geq \left(\frac{1}{\beta} - \frac{1}{r_1^-}\right)\theta\nu_i-\frac{\|e\|_1}{\beta}\geq \left(\frac{1}{\beta} - \frac{1}{r_1^-}\right)\min\{S^{n_1},S^{n_2}\}\min\{\theta^{-\tau_1},\theta^{-\tau_2}\}-\frac{\|e\|_1}{\beta},
	$$ 
	which is a contradiction with \eqref{PS2-c}. Thus, we have proved that $I = \emptyset$. Then, arguing as in the proof Lemma~\ref{Le.PS1} we infer that $u_n \to u$ in $X$ as $n\to \infty$, and the proof is complete.
\end{proof}

Let $\{X_k\}_{k\in\N}$ be a sequence of finitely dimensional spaces of $X$ constructed as in the proof of Lemma~\ref{genus.k} with $B$ given in $(\mathcal{F}_4)$. 
\begin{lemma}\label{Le.Super.Rk}
Let $\theta>0$.	Then, there exists a squence $\{R_k\}_{k\in\N}$ independent of $\theta$ such that $1<R_k<R_{k+1}$ for all $k\in\mathbb{N}$ and for each $k\in\mathbb{N}$,  
	$$\Psi_\theta (u)<0,\quad \forall u \in X_k \text{ with } \|u\|>R_k.$$
\end{lemma}
\begin{proof}
	Let $k\in\N$. By the mutual equivalence of norms in $X_k$, we find $a_k>0$ such that
	\begin{equation}\label{equiv.ak}
		\|u\|_{L^{q^+_B}(B)}\geq a_k \|u\|,\quad \forall u\in X_k.
	\end{equation}
Let $L_k>\left(p^-a_k^{q_B^+}\right)^{-1}$. Then, by ($\mathcal{F}_4$) there exists $T_k>1$ such that:
	$$F(x,t) \geq L_k|t|^{q^+_B},\quad \forall |t|>T_k,\, \text{a.a. } x\in B.$$
	On the other hand, it follows from \eqref{F} that
	$$F(x,t)\geq -2C_1(1+T_k^{\alpha^+}):=N_k,\quad \forall |t|\leq T_k, \, \text{a.a. } x \in \Omega.$$
Combining the last two estimates gives
	\begin{equation}\label{Est.F.2}
		F(x,t) \geq L_k|t|^{q_B^+}-L_kT_k^{q_B^+}+N_k,\quad \forall t\in \mathbb{R}, \text{ a.a. } x \in B.
	\end{equation}
Now, let us choose $R_k>0$ such that 
$$\round{\frac{1}{p^-}- L_ka_k^{q_B^+}}R_k^{q_B^+}+\left(L_kT_k^{q_B^+}-N_k\right)|B|<0.$$
Then, utilizing \eqref{equiv.ak}, \eqref{Est.F.2} and the fact that $\supp u \subset B$ for $u\in X_k$ it holds that
	\begin{align}\label{R_k.1}
		\notag
		\Psi_\theta(u)& \leq \frac{1}{p^-} \|u\|^{q_B^+}-L_ka_k^{q_B^+}\|u\|^{q^+_B}-N_k|B|
		\leq \round{\frac{1}{p^-}-L_ka_k^{q_B^+}}R_k^{q_B^+}+\left(L_kT_k^{q_B^+}-N_k\right)|B|<0
	\end{align}
for all $u \in X_k$ with $\|u\|>R_k$.  Obviously, $\{R_k\}_{k\in\N}$ above can be chosen such that $R_k$ is independent of $\theta$ and $1<R_k<R_{k+1}$ for all $k\in\mathbb{N}$. The proof is complete.  
\end{proof}
	For each $k\in\mathbb{N}$, let $R_k$ be as in Lemma~\ref{Le.Super.Rk} and define
	\begin{equation}
		D_k:=\{u \in X_k:\, \|u\|\leq R_k\},
	\end{equation} 
\begin{equation}\label{Gk}
		G_k:=\{g\in C(D_k,X):\, g\ \text{is odd and}\ g(u)=u\, \text{on}\ \partial D_k\}.
	\end{equation}
	Note that $G_k\ne\emptyset$ since $\operatorname{id}\in G_k.$ Define
	\begin{equation}\label{PT.Super.ck}
		c_k:=\inf_{g\in G_k}\, \max_{u\in D_k}\, \Psi_\theta(g(u)).
	\end{equation}
	Arguing as in \cite[Proof of Theorem 2.1]{AR} we obtain the following via the deformation lemma.
	\begin{lemma}\label{PL.Super.ck}
		For each $k\in \mathbb{N}$, $c_k$ defined in \eqref{PT.Super.ck} is a critical value of $\Psi_\theta$ provided $\Psi_\theta$ satisfies the $\textup{(PS)}_{c_k}$ condition.
	\end{lemma}
Before completing the proof of Theorem~\ref{Theo.sl.M=1}, we introduce a sequence  $\{Z_k\}_{k\in\N}$ of closed linear subspaces of $X$ with finite codimensions, which will be utilized to prove the unboundedness of $\{c_k\}_{k\in\N}$.  Let $\{e_j\}_{j\in\N}\subset X$ be a Schauder basis of $X$. For each $n\in\mathbb{N},$ let $f_n\in X^*$ be defined as
$$\langle f_n,u\rangle=\alpha_n\quad \text{for}\ \ u=\sum_{j=1}^\infty\alpha_je_j\in X.$$
For each $k\in\mathbb{N},$ define
$$Y_k:=\{u\in X:\ \langle f_n,u\rangle=0,\ \ \forall\, n\geq k\}$$
and 
$$Z_k:=\{u\in X:\ \langle f_n,u\rangle=0,\ \ \forall\, n\leq k-1\}.$$
Then, $X=Y_k \oplus Z_k$ and $Z_k$ has codimension $k-1$. Define
\begin{equation}\label{PT.super1.delta_k}
	\delta_k:=\underset{\|v\|\leq 1}{\underset{v\in Z_k}{\sup}}\, \, \|v\|_{\alpha(\cdot)}.
\end{equation}
The next lemma can be proved in the same fashion as \cite[Lemma 4.5]{HHS22}.
\begin{lemma}\label{PL.Super.delta_k}
	The sequence $\{\delta_k\}_{k\in\N}$ above satisfies $0< \delta_{k+1}\leq\delta_k$ for all $k\in\mathbb{N}$ and
	\begin{equation}\label{deltak_to_0.V2}
		\lim_{k\to\infty}\delta_k=0.
	\end{equation} 
\end{lemma}
Finally, we complete the proofs of Theorems~\ref{Theo.sl.M=1} and \ref{Theo2.sl.M=1}.
\begin{proof}[\textbf{Proof of Theorem~\ref{Theo.sl.M=1}}]
	Let $\theta> 0$ and define $\{c_k\}_{k\in\N}$ as in \eqref{PT.Super.ck}. Let $k\in\N.$ We will obtain an upper bound and a lower bound for $c_k$ as follows. Since $\operatorname{id}\in G_k,$ the definition \eqref{PT.Super.ck} of $c_k$ yields
	\begin{equation}\label{PT.Super.Est-ck-1}
		c_k\leq \max_{u\in D_k}\, \Psi_\theta(u).
	\end{equation}
On the other hand, by means of the estimate \eqref{F} and Propositions~\ref{prop.nor-mod.D}-\ref{prop_embs}, for all $u \in D_k$ we have
\begin{align*}
	\notag
	\Psi_\theta(u)&\leq \frac{1}{p^-}\int_\Omega \mathcal{H}(x,|\nabla u|)\diff x-\int_{\Omega}F(x,u)\diff x\\
	\notag&\leq \frac{1}{p^-}\round{\|u\|^{q^+}+1}+2C_1\int_\Omega \left(|u|^{\alpha(x)}+1\right)\diff x\\
	\notag&\leq \frac{1}{p^-}\|u\|^{q^+}+2C_1\left(\|u\|_{\alpha(\cdot)}^{\alpha^+}+1\right)+2C_1|\Omega|+\frac{1}{p^-}\\
\notag	&\leq \frac{1}{p^-}\|u\|^{q^+}+2C_9\|u\|^{\alpha^+}+2C_1\left(|\Omega|+1\right)+\frac{1}{p^-}\\
	&\leq  \frac{1}{p^-}R_k^{q^+}+2C_9R_k^{\alpha^+}+2C_1\left(|\Omega|+1\right)+\frac{1}{p^-}=:C_*(R_k).
\end{align*}
Combining this with \eqref{PT.Super.Est-ck-1} gives
\begin{equation}\label{PT.Super.Est-ck.L}
	c_k \leq C_*(R_k).
\end{equation}
To get a lower bound for $c_k$, we make use of the following fact, which is a consequence of \cite[Lemma 3.9]{KK16}:
$$g(D_k)\cap \partial B_\tau \cap Z_k \neq \emptyset\, \text{ for all } g\in G_k \text{ and all } \tau \in (0,R_k).$$
Thus, for any $\tau \in (0,R_k)$, we have
\begin{equation*}
\max_{u\in D_k}\Psi_\theta (g(u)) \geq\, \inf_{u\in \partial B_\tau\cap Z_k}\Psi_\theta(u),\quad \forall g\in G_k.
\end{equation*}
Let $\tau\in (1,R_k)$ be arbitrary and fixed. From the last inequality and the definition \eqref{PT.Super.ck} of $c_k$ we obtain
\begin{equation}\label{PT.Super.low.est.ck}
	c_k\geq\, \inf_{u\in \partial B_\tau\cap Z_k}\Psi_\theta(u).
\end{equation}
By means of \eqref{F} and Propositions~\ref{prop.nor-mod.D} again, for any $u\in \partial B_\tau\cap Z_k,$ we have
\begin{align}\label{PT.Super.est1I}
	\notag
	\Psi_\theta(u)&\geq\frac{1}{q^+}\left(\|u\|^{p^-}-1\right)-2C_1\int_\Omega |u|^{\alpha(x)}\diff x-2C_1|\Omega| -\frac{\theta}{r_1^-}\int_\Omega \mathcal{B}(x,u)\diff x\\ 
	&\geq \frac{1}{q^+}\|u\|^{p^-}-2C_1\|u\|_{\alpha(\cdot)}^{\alpha^+}-\frac{\theta}{r_1^-}\int_\Omega \mathcal{B}(x,u)\,\diff x -2C_1(|\Omega|+1)-\frac{1}{q^+}.
	\end{align}
Note that for $\|u\|=\tau>1$ it follows from  Proposition~\ref{prop.nor-mod.D} and \eqref{S} that
\begin{equation}\label{PT.Super.est1I'}
	\int_\Omega \mathcal{B}(x,u)\diff x\leq \left(1+S^{-r_2^+}\right)\|u\|^{r_2^+}.
\end{equation}
 Because of \eqref{deltak_to_0.V2}, we can find  $k_0\in\mathbb{N}$ such that $\delta_k<1<\frac{p^-}{4C_1q^+\alpha^+\delta_k^{\alpha^+}}$ for all $k\geq k_0.$ Moreover, \eqref{PT.super1.delta_k} implies that
$$\|u\|_{\alpha(\cdot)} \leq \delta_k \|u\|,\quad \forall u\in Z_k.$$
 Combining this with \eqref{PT.Super.est1I}-\eqref{PT.Super.est1I'} we infer that for all $k\geq k_0,$
\begin{align*}\label{PT.Super.est2I}
	\Psi_\theta(u)\geq\frac{1}{q^+}\tau^{p^-}-2C_{1}\delta_k^{\alpha^+}\tau^{\alpha^+} -\frac{ \left(1+S^{-r_2^+}\right)\theta}{r_1^-}\tau^{r_2^+}-2 C_1(|\Omega|+1)-\frac{1}{q^+},\quad \forall u\in \partial B_\tau\cap Z_k.
\end{align*}
 From this and \eqref{PT.Super.low.est.ck} we obtain that for $k\geq k_0,$
\begin{equation}\label{PT.Super.est3I}
	c_k\geq \frac{1}{q^+}\tau^{p^-}-2C_{1}\delta_k^{\alpha^+}\tau^{\alpha^+} -\frac{ \left(1+S^{-r_2^+}\right)\theta}{r_1^-}\tau^{r_2^+}-2C_1(|\Omega|+1)-\frac{1}{q^+},\quad \forall \tau\in (1,R_k).
\end{equation}
Note that for $\theta$ satisfying
\begin{equation*}\label{PT.Super.theta}
	0< \theta\leq \frac{r_1^-}{2q^+(1+S^{-r_2^+})}R_k^{p^--r_2^+}=:C^\ast(R_k),
\end{equation*}
we get
\begin{equation*}
	\frac{ \left(1+S^{-r_2^+}\right)\theta}{r_1^-}\tau^{r_2^+}\leq \frac{1}{2q^+}\tau ^{p^-},\quad \forall \tau \in (1,R_k).
\end{equation*}
Then, \eqref{PT.Super.est3I} implies that for any $k\geq k_0,$ 
\begin{align}\label{PT.Super.Est-ck}
	c_k&\geq\frac{1}{2q^+}\tau^{p^-}-2C_{1}\delta_k^{\alpha^+}\tau^{\alpha^+}-2C_1(|\Omega|+1)-\frac{1}{q^+}=:h(\tau),\quad \forall \tau\in (1,R_k).
\end{align}
By choosing $R_k$ in Lemma~\ref{Le.Super.Rk} larger if necessary, we may assume that for each $k\in\N$, $R_k>\tau_0:=\left(\frac{p^-}{4C_1q^+\alpha^+\delta_k^{\alpha^+}}\right)^{\frac{1}{\alpha^+-p^-}}$. Hence,
 for any $k\geq k_0,$ \eqref{PT.Super.Est-ck} implies that 
\begin{align}\label{PT.Super.Est-ck.U}
	c_k\geq& \max\limits_{\tau \in (1,R_k)}h(\tau)=h(\tau_0)=a_0\delta_k^{-\frac{\alpha^+ p^-}{\alpha^+-p^-}}+b_0,
\end{align}
provided $\theta\in (0,C^\ast(R_k))$,
where
$$a_0:=\left(\frac{p^-}{4C_1q^+\alpha^+}\right)^{\frac{p^-}{\alpha^+-p^-}}\frac{\alpha^+-p^-}{2q^+\alpha^+}>0\ \text{ and } \ b_0:=-2C_1(|\Omega|+1)-\frac{1}{q^+}.$$

We now prove the conclusion of Theorem~\ref{Theo.sl.M=1} by picking a sequence $\{\theta_n\}_{n\in\N}$ as follows. By Lemma~\ref{PL.Super.delta_k}, we find $k_1\geq k_0$ such that
\begin{equation*}
a_0\delta_{k_1}^{-\frac{\alpha^+ p^-}{\alpha^+-p^-}}+b_0>0.
\end{equation*}
Then, choose $\theta_1$ satisfying
\begin{equation}\label{theta1}
	\begin{cases}
		\theta_1\in (0, C^\ast(R_{k_1})]\\
		C_\ast(R_{k_1})<\left(\frac{1}{\beta} - \frac{1}{r_1^-}\right)\min\{S^{n_1},S^{n_2}\}\min\{\theta_1^{-\tau_1},\theta_1^{-\tau_2}\}-\frac{\|e\|_1}{\beta}.
	\end{cases}
\end{equation}
Note that the choice of $\theta_1$ and \eqref{PT.Super.Est-ck.U} yield
\begin{equation}\label{PT.Super.Est-ck1}
	c_{k_1}\geq a_0\delta_{k_1}^{-\frac{\alpha^+ p^-}{\alpha^+-p^-}}+b_0>0,\quad \forall \theta\in (0,\theta_1),
\end{equation}
and $\Psi_\theta$ satisfies $\textup{(PS)}_{c_{k_1}}$ thanks to \eqref{PT.Super.Est-ck.L}, \eqref{theta1} and Lemma~\ref{Le.PS4.c}.
Inductively, for $\theta_n$ satisfying
\begin{equation}\label{PT.Supper.theta_n}
	\begin{cases}
	\theta_n \in (0,C^*(R_{k_n})], \\
	C_*(R_{k_n})<\left(\frac{1}{\beta} - \frac{1}{r_1^-}\right)\min\{S^{n_1},S^{n_2}\}\min\{\theta_n^{-\tau_1},\theta_n^{-\tau_2}\}-\frac{\|e\|_1}{\beta},
	\end{cases}
\end{equation}
we take $k_{n+1}\in\mathbb{N}$, $k_{n+1}>k_n$ such that
\begin{align}\label{PT.Super.delta_k_n+1}
a_0\delta_{k_{n+1}}^{-\frac{\alpha^+ p^-}{\alpha^+-p^-}}+b_0>\left(\frac{1}{\beta} - \frac{1}{r_1^-}\right)\min\{S^{n_1},S^{n_2}\}\min\{\theta_n^{-\tau_1},\theta_n^{-\tau_2}\}-\frac{\|e\|_1}{\beta},
\end{align}
and then pick $\theta_{n+1}\in (0,\theta_{n})$ satisfying
\begin{equation}\label{PT.Supper.theta_n+1}
	\begin{cases}
		\theta_{n+1}\in (0, C^*(R_{k_{n+1}})]\\
	C_*(R_{k_{n+1}})<\left(\frac{1}{\beta} - \frac{1}{r_1^-}\right)\min\{S^{n_1},S^{n_2}\}\min\{\theta_{n+1}^{-\tau_1},\theta_{n+1}^{-\tau_2}\}-\frac{\|e\|_1}{\beta}.
	\end{cases}
\end{equation}
 Now, let $\theta\in (0, \theta_n)$ for some $n\in\mathbb{N}.$ By \eqref{PT.Super.Est-ck.L} and \eqref{PT.Super.Est-ck.U}-\eqref{PT.Supper.theta_n+1} we have
\begin{equation*}
	0<c_{k_1}<c_{k_2}<\cdots<c_{k_n}\leq C_*(R_{k_n}).
\end{equation*} 
From this and \eqref{PT.Supper.theta_n} we arrive at
\begin{equation*}
	0<c_{k_1}<c_{k_2}<\cdots<c_{k_n}<\left(\frac{1}{\beta} - \frac{1}{r_1^-}\right)\min\{S^{n_1},S^{n_2}\}\min\{\theta^{-\tau_1},\theta^{-\tau_2}\}-\frac{\|e\|_1}{\beta}.
\end{equation*} 
Thus, $c_{k_1},\, c_{k_2},\, \cdots,\, c_{k_n}$ are distinct critical values of  $\Psi_\theta$ in view of Lemma~\ref{Le.PS4.c} and Lemma~\ref{PL.Super.ck}. Hence,  $\Psi_\theta$ has at least $n$ distinct pairs of critical points. The proof is complete.
\end{proof}

\begin{proof}[\textbf{Proof of Theorem~\ref{Theo2.sl.M=1}}] The proof is almost identical with that of  Theorem~\ref{Theo.sl.M=1} above. The only difference is that thanks to $(\mathcal{F}_5)$, we do not need to make use of $(\mathcal P)$ in the proof of Lemma~\ref{Le.PS4.c}.
	
\end{proof}

	\vspace{0.5cm}
	\noindent{\bf Acknowledgment}\\
	
	%\medskip
	Ky Ho was partially supported by the National Research Foundation of Korea (NRF) grant funded by the Korea government (MSIT) (grant No. 2022R1A4A1032094).

\vspace{0.5cm}
\noindent{\bf Declaration of generative AI and AI-assisted technologies in the writing process}\\	
	
	During the preparation of this work the authors used QuillBot in order to improve language of some sentences in the Introduction. After using this tool, the authors reviewed and edited the content as needed and take full responsibility for the content of the publication.

\end{document}